\documentclass[10pt,a4paper]{amsart}

\usepackage{amsthm, amsfonts, amssymb, amsmath, latexsym, enumerate,array}
\usepackage[all]{xy}
\usepackage[utf8]{inputenc}
\usepackage{hyperref,cite,mdwlist,mathrsfs}
\usepackage{tikz}
\usepackage{mathtools}
\usepackage{braids}
\usepackage{caption}
\usepackage{subcaption}
\usepackage[sans]{dsfont}
\usepackage{etex}
\usepackage{todonotes}

\usepackage{mdwlist}
\usepackage[left=1.35in,top=0.8in,right=1.35in,bottom=0.8in]{geometry}

\usepackage[bitstream-charter]{mathdesign}

\newtheorem{thm}{Theorem}[section]
\newtheorem{lem}[thm]{Lemma}
\newtheorem{corol}[thm]{Corollary}
\newtheorem{prop}[thm]{Proposition}

\newtheorem{main}{Theorem}

\newtheorem*{thm*}{Theorem}
\newtheorem*{corol*}{Corollary}
\newtheorem*{cnj*}{Conjecture}

\theoremstyle{definition}
\newtheorem{rmk}[thm]{Remark}
\newtheorem{eg}[thm]{Example}

\newtheorem{step}{Step}

\newcommand{\cA}{\mathscr{A}}
\newcommand{\cB}{\mathscr{B}}

\newcommand{\cH}{\mathscr{H}}
\newcommand{\cGG}{\mathscr{G}}
\newcommand{\cE}{\mathscr{E}}
\newcommand{\cU}{\mathscr{U}}
\newcommand{\cR}{\mathscr{R}}
\newcommand{\cK}{\mathscr{K}}
\newcommand{\cC}{\mathscr{C}}
\newcommand{\cL}{\mathscr{L}}
\newcommand{\cF}{\mathscr{F}}
\newcommand{\cG}{\mathfrak{g}}
\newcommand{\cV}{\mathscr{V}}
\newcommand{\cW}{\mathscr{W}}
\newcommand{\cT}{\mathscr{T}}
\newcommand{\cN}{\mathscr{N}}
\newcommand{\cS}{\mathfrak{s}}

\newcommand{\cP}{\mathscr{P}}

\newcommand{\cO}{\mathscr{O}}
\newcommand{\RHom}{\mathrm{RHom}}

\DeclareMathOperator{\rk}{rk}
\DeclareMathOperator{\id}{id}
\DeclareMathOperator{\Ext}{Ext}
\DeclareMathOperator{\Hom}{Hom}

\DeclareMathOperator{\ext}{ext}

\DeclareMathOperator{\HH}{H}
\DeclareMathOperator{\hh}{h}

\DeclareMathOperator{\rL}{L}
\DeclareMathOperator{\rR}{R}

\DeclareMathOperator{\codim}{codim}

\DeclareMathOperator{\fC}{{\mathfrak c}}
\DeclareMathOperator{\fB}{{\mathfrak b}}
\DeclareMathOperator{\fZ}{{\mathfrak z}}
\DeclareMathOperator{\fJ}{{\mathscr J}}
\DeclareMathOperator{\fF}{{\mathfrak f}}
\DeclareMathOperator{\fH}{{\mathfrak h}}
\DeclareMathOperator{\fU}{{\mathfrak u}}

\newcommand{\Z}{\mathds Z}
\newcommand{\FF}{\mathds F}

\newcommand{\N}{\mathds N}
\newcommand{\PP}{\mathds P}

\newcommand{\PD}{\check{\mathds P}}

\DeclareMathOperator{\ts}{\otimes}

\newcommand{\epi}{\twoheadrightarrow}
\newcommand{\xr}{\xrightarrow}

\newcommand{\bk}{{\mathds{k}}}

\newcommand{\bP}{\mathds{P}^1}
\newcommand{\bp}{\boxplus}
\newcommand{\sx}{{\boldsymbol x}}
\newcommand{\sy}{{\boldsymbol y}}
\newcommand{\sa}{{\boldsymbol u}}
\newcommand{\bD}{\mathbf{D}}
\newcommand{\bC}{\mathbf{C}}
\newcommand{\bB}{\mathbf{B}}

\newcommand{\bR}{\mathbf{R}}

\newcommand{\coker}{\mathrm{coker}}
\newcommand{\depth}{\mathrm{depth}}
\newcommand{\notA}{B}
\newcommand{\notB}{A}
\newcommand{\vk}{{\vec k}}

\newcommand{\aaa}{a} \newcommand{\bb}{b} \newcommand{\cc}{c} \newcommand{\dd}{d}
\newcommand{\fbK}{{\boldsymbol{\mathfrak K}}}

\allowdisplaybreaks[1]
\numberwithin{equation}{section}

\begin{document}


\title{Surfaces of minimal degree of tame representation type and mutations of Cohen-Macaulay modules}

\author{Daniele Faenzi}
\email{\tt daniele.faenzi@u-bourgogne.fr}
\address{Université de Bourgogne Franche-Comté\\
  9 Avenue Alain Savary - BP 47870 - 21078 Dijon Cedex - France}

\author{Francesco Malaspina}
\email{{\tt francesco.malaspina@polito.it}}
\address{Politecnico di Torino \\
  Corso Duca degli Abruzzi 24 - 10129 Torino - Italy}



\thanks{D. F. partially supported by GEOLMI ANR-11-BS03-0011. F. M. partially supported
  by GRIFGA and PRIN 2010/11
{\it Geometria delle varietà algebriche}, cofinanced by MIUR}

\begin{abstract}
We provide two examples of smooth projective surfaces of tame CM type,
by showing that the parameter space of isomorphism classes of
indecomposable ACM bundles with fixed rank and determinant on a
rational quartic scroll in $\PP^5$ is either a single point or a
projective line. 
These turn out to be  the only smooth projective varieties of tame CM
type besides elliptic curves, \cite{atiyah:elliptic}.

For surfaces of minimal degree and wild CM type, we classify rigid Ulrich bundles as Fibonacci extensions.
For $\FF_0$ and $\FF_1$, embedded as quintic or sextic scrolls, a 
complete classification of rigid ACM bundles is given.

\end{abstract}

\maketitle

\section*{Introduction}

Let $X \subset \PP^n$ be a smooth positive-dimensional closed subvariety
over an algebraically closed field $\bk$, and
assume that the graded coordinate ring $\bk[X]$ of $X$ is Cohen-Macaulay, i.e. $X$
is ACM (arithmetically Cohen-Macaulay).
Then $X$ supports infinitely many indecomposable ACM sheaves $\cE$
(i.e. whose $\bk[X]$-module of global sections $\HH^0_*(X,\cE)$ is Cohen-Macaulay), unless 
$X$ is $\PP^n$ itself, or a quadric hypersurface, or a rational normal
curve, or one of the two sporadic cases:
the Veronese surface in $\PP^5$ and (cf. \S
\ref{hirzebruch}) the rational cubic scroll $S(1,2)$ in
$\PP^4$, see \cite{eisenbud-herzog:CM}.

Actually, for most ACM varieties $X$, much more is true. Namely $X$ supports families of arbitrarily
large dimension of indecomposable ACM bundles, all non-isomorphic to
one another (varieties like this are of ``geometrically of wild CM type''
or simply ``CM-wild'').
CM-wild varieties include curves of genus $\ge 2$, hypersurfaces of degree $d \ge 4$ in $\PP^n$ with $n \ge 2$,
complete intersections in $\PP^n$ of codimension 
$\ge 3$, having one defining polynomial of degree  $\ge 3$
(cf. \cite{crabbe-leuschke,drozd-tovpyha:wild}), 
the third Veronese embedding of any variety of dimension $\ge 2$ cf. \cite{miro_roig:veronese-PAMS}.
In many cases, these
families are provided by
Ulrich bundles, i.e. those $\cE$ 
such that $\HH^0_*(X,\cE)$ achieves the maximum 
number of generators,
namely $ d_X\rk(\cE)$, where we write $d_X$ for the degree of $X$.
For instance, Segre embeddings are treated in \cite{costa-miro_roig-pons_llopis},
smooth rational ACM surfaces in $\PP^4$ in \cite{miro_roig-pons_llopis:surfaces_P4},
cubic surfaces and threefolds in \cite{casanellas-hartshorne:cubic-JEMS,casanellas-hartshorne-geiss-schreyer}, 
del Pezzo surfaces in \cite{pons_llopis-tonini,coskun-kulkarni-mustopa:ulrich}.

In spite of this, there is a special class of varieties $X$ with
intermediate behaviour, namely $X$ supports continuous families of
indecomposable ACM bundles, all non-isomorphic to one another, but, for each rank $r$, these bundles form
finitely (or countably) many irreducible families of dimension at most one.
Then $X$ is called of tame CM type. It is the case of the elliptic curve, \cite{atiyah:elliptic}.

In this note we provide the first examples of smooth
positive-dimensional projective CM-tame varieties, besides elliptic curves. Part of
this was
announced in \cite{faenzi-malaspina:CRAS}.

\begin{main} \label{main-tame}
Let $X$ be 
a smooth  surface of degree $4$ in $\PP^5$. Then, for any $r \ge 1$, there is a
family of isomorphism classes of indecomposable Ulrich bundles of rank
$2r$, parametrized by $\PP^1$. Conversely, any indecomposable
ACM bundle on $X$ is rigid or belongs to one of these families (up to a twist).
In particular, $X$ is of tame CM type. 
\end{main}

Recalling the classification by del Pezzo and
Bertini  of smooth varieties of {\it minimal degree},
i.e. with $d_X=\codim(X)+1$,
as the Veronese surface in $\PP^5$ and rational normal scrolls
(cf. \cite{eisenbud-harris:centennial}), we see two things. On one hand, 
both CM-finite varieties and our examples have minimal degree,
actually a surface of degree $4$ in $\PP^5$ is a quartic
scroll. Incidentally, these have the same graded
Betti numbers as the Veronese surface in $\PP^5$, which is CM-finite.
On the
other hand the remaining varieties of minimal degree are CM-wild
 by \cite{miro-roig:scrolls}; in fact in
\cite{miro-roig:scrolls} also quartic scrolls are claimed to be of
wild CM type: the gap in the argument
only overlooks our two examples, cf. Remark \ref{occhioallerrore}.
By the following result, these examples complete the list of
non-CM-wild varieties in a broad sense.

\begin{thm*}[cf. \cite{faenzi-pons:arxiv}]
  Let $X \subset \PP^n$ be a closed normal ACM subvariety
  of positive dimension, which is not a cone. Then $X$ is CM-wild if it
  is not one of the 
  well-known CM-finite
  varieties, or an elliptic curve, or a quartic surface scroll.
\end{thm*}

For CM-wild varieties, an interesting issue is to study rigid ACM
bundles (by definition $\cE$ is rigid if $\Ext^1_X(\cE,\cE)=0$).
Important instances of classification of such bundles
are given in \cite{iyama-yoshino,keller-murfet-van_den_bergh},
cf. also \cite{faenzi:iyama-yoshino}, for the third Veronese surface
and second Veronese threefold, both embedded in $\PP^9$.

Our first result in this direction deals with Ulrich bundles.
Given a rational surface scroll
$X$, we set $H$ for the hyperplane class, $F$ for the
class of a fibre of the projection $X \to \PP^1$, and $\cL=\cO_X((d_X-1)F-H)$.
For $w\ge 2$ we define the Fibonacci numbers
by the relations $\phi_{w,0}=0$, $\phi_{w,1}=1$ and 
$\phi_{w,k+1}=w \phi_{w,k}-\phi_{w,k-1}$. We set $\phi_{w,-1}=0$.

\begin{main} \label{main-wild}
Let $X$ be a smooth surface scroll of degree $d_X \ge
5$. Set $w=d_X-2$. Then:
\begin{enumerate}[(i)]
\item \label{parte_estensioni}  an indecomposable ACM bundle $\cE$ on
  $X$ is Ulrich if and only if,  up to twist, $\cE$ fits into:
 \[
 0 \to \cO_X(-F)^{a} \to \cE \to \cL^{b} \to 0, \qquad \mbox{for some
   $a,b \ge 0$;}
 \]

\item \label{primaparte}  if moreover $\cE$ is rigid then
  $a=\phi_{w,k}$ and
 $b=\phi_{w,k\pm 1}$, for some $k\ge 0$;
\item \label{secondaparte} for any $k \in \Z$, there is a unique
  indecomposable rigid Ulrich bundle $\fU_k$
   with:
  \begin{align*}
  & a=\phi_{w,k}, && b=\phi_{w,k-1}, && \mbox{for $k\ge 1$},\\
  & a=\phi_{w,-k}, && b=\phi_{w,1-k}, && \mbox{for $k\le 0$}.
  \end{align*}
  Furthermore, each $\fU_k$ is a exceptional, and $\fU_{k}(-H) \simeq \fU_{1-k}^* \ts
  \omega_X$.
\end{enumerate}
\end{main}

Here, a bundle $\cE$ is  {\it exceptional} if
$\RHom_X(\cE,\cE)=\langle \id_\cE\rangle$. ``Up to twist''
means ``up to tensoring with $\cO_X(tH)$ for
some $t \in \Z$''. The following result should be compared with \cite{casanellas-hartshorne:cubic-JEMS}.  

\begin{corol*}
  There is no stable Ulrich bundle of rank greater than one on a surface of
  minimal degree. If this degree is $4$, any non-Ulrich indecomposable
  ACM bundle is rigid.
\end{corol*}

Our next result concerning the classification of rigid ACM sheaves 
 deals with the CM-wild  scrolls
$S(\vartheta-1,\vartheta)$ and $S(\vartheta,\vartheta)$ for $\vartheta
\ge 3$.
To state it we anticipate from cf. \S \ref{braid-definition} and \S \ref{rigid construction}.
Consider the braid group $B_3$, whose standard
generators $\sigma_1$ and $\sigma_2$ act
by right mutation over $3$-terms exceptional collections $(\cS_1,\cS_2,\cS_3)$ over $X$,
starting with $\bB_\emptyset=(\cL[-1],\cO_X(-F),\cO_X)$.
Given a vector
$\vk=(k_1,\ldots,k_s) \in \Z^s$, we let $\sigma^{\vk} = \sigma_1^{k_1}
\sigma_2^{k_2} \sigma_1^{k_3} \cdots$, and $\sigma^\emptyset = 1$.
The exceptional collection $\bB_\vk$ obtained by $\sigma^\vk$ can be extended to
a full exceptional collection $\bC_\vk=(\cL(-F)[-1],\bB_\vk)$.
Thinking of the Euler characteristic $v_{i-1}:=\chi(\cS_{i},\cS_{i+1})$
(with cyclic indexes) of pairs of bundles in the
mutated $3$-term collection $\bB_\vk$, we  
define an operation of $B_3$ on $\Z^3$, by:
\begin{align*}
  \sigma_1 : (v_1,v_2,v_3) \mapsto (v_1v_3-v_2,v_1,v_3), &&
  \sigma_2 : (v_1,v_2,v_3) \mapsto (v_1,v_1v_2-v_3,v_2).
\end{align*}
To $\bB_\emptyset$ corresponds $v^\emptyset = (2,d_X-4,d_X-2)$.
Set $\bar t \in \{0,1\}$ for the remainder of the
division of an integer $t$ by $2$.
Given $\vk=(k_1,\ldots,k_s)$ and $t\le s$, we write the truncation
$\vk(t)=(k_1,\ldots,k_{t-1})$.
Having this set up, we finally define the set:
\[  \fbK = \{\vk = (k_1,\ldots,k_s) \mid (-1)^{t} k_{t-1}
(\sigma^{\vk(t)}
.v^\emptyset)_{2 \bar t +1} \le 0, \forall t \le s\}.
\]
Explicitly, a vector length-$s$ vector $\vk=(k_1,\ldots,k_s)$ belongs
to $\fbK$ if, for all subvectors $\vk(t)=(k_1,\ldots,k_{t-1})$,
applying $\sigma^{\vk(t)}$ to $v$ we get a triple of integer whose
\begin{itemize}
\item third element has the same sign as $k_{t-1}$ (for odd $t$);
\item first element has opposite sign with respect to $k_{t-1}$ (for even $t$).
\end{itemize}

\begin{main} \label{alla fine je l'abbiamo fatta}
  Let $\vk$ be an element of $\fbK$. Then, there is an exceptional ACM
  bundle $\fF_\vk$ corresponding to $\vk$,
  which is the middle element of the exceptional collection $\bB_\vk$.
  If $\vartheta=3$, any indecomposable rigid ACM
  bundle is of the form $\fF_\vk$ or $\fF_\vk^* \ts \omega_X$, for some $\vk \in
  \fbK$, up to twist.
\end{main}

The structure of the paper is as follows.
In \S \ref{section background}  we write the basic form of
a resolution of ACM bundles on scrolls relying on the structure of the
derived category of a $\PP^1$-bundle. We also provide here some
cohomological splitting criteria for scrolls of low degree. 
In \S \ref{section kronecker} we
study Ulrich bundles on surfaces of minimal degree in terms of representations of
Kronecker quivers and prove Theorem \ref{main-wild} using Kac's
classification of Schur roots.
In \S \ref{quartic} we focus on quartic scrolls
and prove their tameness 
according to Theorem \ref{main-tame}.
In \S \ref{rigid section}, we give the proof of Theorem \ref{alla fine je l'abbiamo fatta}.

\section{Resolutions for bundles on scrolls}

In this section, after providing some essential
terminology, we give our basic technique to
classify ACM bundles on ruled surfaces. Indeed, it is shown in \S\ref{ab}
that such bundle $\cE$ has a functorial two-sided resolution computed upon
certain cohomology groups of $\cE$. We derive in
\S\ref{section splitting} a splitting criterion over scrolls
of low degree.

\label{section background}

\subsection{Background}

Let $\bk$ be an algebraically closed field.
Given a vector space $V$ over $\bk$, we let $\PP V$ be the
projective space of $1$-dimensional quotients of $V$.
If $\dim(V)=n+1$, we write $\PP^n=\PP V$.
It will be understood that a small letter denotes the dimension of a vector space
 in capital letter, for instance if $\cE$ is a coherent sheaf
on a variety $X$ then $\hh^i(X,\cE) = \dim _\bk \HH^i(X,\cE)$.
For a pair of coherent sheaves $\cE_i$, $\cE_2$ on $X$, the
Euler characteristic is
$\chi(\cE_1,\cE_2)=\sum(-1)^j\ext^j_X(\cE_1,\cE_2)$. We abbreviate
$\chi(\cO_X,\cE)$ to $\chi(\cE)$. 
A {\it (vector) bundle} is a  coherent locally free sheaf.

\subsubsection{Derived categories}

We will use the derived category $\bD^b(X)$ of bounded complexes of
coherent sheaves on $X$.
We refer to \cite{huybrechts:fourier-mukai} for a detailed account of it.
An object $\cE$ of $\bD^b(X)$ is a bounded complex of coherent sheaves
on $X$, we will denote by $\HH^\bullet(X,\cE)$ the total complex
associated with the hypercohomology of $\cE$.
If $\cE$ is concentrated in degree $i$, then we write $|\cE|$ for the
coherent sheaf $\cH^i(\cE)$, i.e. $|\cE| = \cE[i]$.

An object $\cE$ of  $\bD^b(X)$ is {\it simple} if 
$\Hom_X(\cE,\cE) \simeq \bk$, and {\it exceptional} if
$\RHom_X(\cE,\cE) \simeq \bk$.
Given a set $S$ of objects of $\bD^b(X)$, we write $\langle S \rangle$
for the smallest full triangulated subcategory of $\bD^b(X)$
containing all objects of $S$. The same notation is used for a
collection $S$ of subcategories of $\bD^b(X)$.
Given a pair of objects $\cE$ and $\cF$ of $\bD^b(X)$ we have the left
and right mutations $\rL_\cE(\cF)$ and $\rR_\cF(\cE)$ defined
respectively by the distinguished triangles given by natural evaluations:
\[
\RHom_X(\cE,\cF) \ts \cE \to \cF \to \rL_\cE(\cF)[1], \qquad
\rR_\cF(\cE)[-1] \to \cE \to \RHom_X(\cE,\cF)^* \ts \cF.
\]

\subsubsection{Arithmetically Cohen-Macaulay varieties and sheaves}
 A {\it polarized} or {\it embedded variety} is a pair $(X,H)$, where
$X$ is an integral $m$-dimensional projective variety and $H$ is a {\it very ample} divisor class on
$X$. The degree $d_X$ is $H^{m}$. If
$H$ embeds $X$ in $\PP^n$, then the ideal $I_X$ of $X$ sits in
$R = \bk[x_0,\ldots,x_n]$ and the homogeneous coordinate $\bk[X]$ is $R/I_X$.
The variety $X \subset \PP^n$ is {\it ACM} (for arithmetically Cohen-Macaulay) if
$\bk[X]$ is a graded Cohen-Macaulay ring, i.e.  the $R$-projective dimension of $\bk[X]$ 
is $n-m$.

Given a coherent sheaf $\cE$ on a polarized variety $(X,H)$, and
$i,t \in \N$, we write $\cE(t)$ for $\cE(t H)$ and:
\[
\HH^i_*(X,\cE)=\bigoplus_{t \in \Z} \HH^i(X,\cE(t)).
\]
For each $i$, $\HH^i_*(X,\cE)$ is a module over $\bk[X]$.
We say that $\cE$ is {\it initialized} if $\HH^0(X,\cE) \ne 0$ and
$\HH^0(X,\cE(-1)) = 0$, i.e. if $\HH^0_*(X,\cE)$ is zero in negative
degrees and non-zero in positive degrees.
Any torsion-free sheaf $\cE$ on a positive-dimensional variety has an {\it initialized twist},
i.e. there is a unique integer $t_0$ such that $\cE(t_0)$ is initialized.
The $\bk[X]$-module $\HH^0_*(X,\cE)$ is also finitely generated in this case.

Given $m \ge 1$, a vector bundle $\cE$ on an smooth ACM $m$-dimensional variety $X$ is {\it ACM} (for arithmetically Cohen-Macaulay) if $\cE$ has no
intermediate cohomology:
\[
\HH^i_*(X,\cE)=0, \qquad \mbox{for all $1 \le i \le m-1$}.
\]
Equivalently, $\cE$ is ACM if $E=\HH^0_*(X,\cE)$ is a maximal Cohen-Macaulay
module over $\bk[X]$, i.e.,  $\depth(E) = \dim(E) = m+1$.

The initialized twist $\cE(t_0)$ of $\cE$ satisfies $\hh^0(X,\cE(t_0)) \le d_X \rk(\cE)$.
We say that $\cE$ is {\it Ulrich} if equality is attained in
the previous inequality.

We will use Gieseker-Maruyama and slope-semi-stability of bundles with
respect to a given polarization. We refer to
 \cite{huybrechts-lehn:moduli}.

\subsection{Reminder on Hirzebruch surfaces and their derived categories}

\label{hirzebruch}

Let $U$ be a $2$-dimensional $\bk$-vector space, and let
$\PP^1=\PP U$ so that $U=\HH^0(\PP^1,\cO_{\PP^1}(1))$. There is a
 identification $U \simeq U^*$, canonical up to the choice of a nonzero scalar.
Let $\epsilon \ge 0$ be an integer and consider the Hirzebruch surface
$\FF_\epsilon = \PP(\cO_{\PP^1}\oplus \cO_{\PP^1}(\epsilon))$.
Write $\pi : \FF_\epsilon \to \PP^1$ for the projection on the base,
 let $F = c_1(\pi^*(\cO_{\PP^1}(1)))$ be the  class of a
fibre of $\pi$, and $\cO_\pi(1)$ be the relatively ample tautological
line bundle on $\FF_\epsilon$.
For any integer $\vartheta > 1$, setting $d_X=2\vartheta+\epsilon$, the surface $\FF_\epsilon$ is embedded in
$\PP^{d_X + 1}$ by the line bundle $\cO_\pi(1) \ts
\cO_X(\vartheta F)$, as a ruled surface of degree $d_X$.
We denote by $\cO_X(H)$ this line bundle, so that the polarized
variety $(\FF_\epsilon,H)$ is a rational normal scroll $S(\vartheta,\vartheta+\epsilon)$.
Of course, we have $F^2=0$ and $F \cdot H=1$.
The canonical bundle of $X=\FF_\epsilon$ is $\omega_X \simeq \cO_X((d_X-2)F-2H)$.

For $\epsilon > 0$, denote by $\Delta \in |\cO_X(H-(\vartheta+\epsilon)F)|$ the {\it negative section} of $X$,
i.e. the section of
$\pi$ with self-intersection $-\epsilon$.
We have for $a \ge 0$:
\[
\mbox{$\hh^k(X,\cO_X(a H+b F))=\sum_{i=0,\ldots,a}\hh^k(\PP^1,\cO_{\PP^1}(a\vartheta+i\epsilon
+ b)).$}
\]
On the other hand $\HH^k(X,\cO_X(b F-H))=0$ for all $k$, while
$\hh^k(X,\cO_X(a H+b F))$ can be computed for $a \le -2$ by Serre duality.
We fix the notation:
\[
\cL = \cO_X((d_X-1)F-H).
\]

\subsubsection{Derived category of Hirzebruch surfaces}

By a result of Orlov \cite{orlov:bundles}, we have the semiorthogonal
decomposition (for a definition cf. for instance \cite[\S 1]{huybrechts:fourier-mukai}):
\begin{equation}
  \label{orlov}
\bD^b(X)=\langle \pi^*\bD^b(\PP^1) \ts \cO_X(-H), \pi^*\bD^b(\PP^1) \rangle.  
\end{equation}
In turn, by Beilinson's theorem, see for instance
\cite[\S 8]{huybrechts:fourier-mukai}, we have:
\begin{equation}
  \label{beilinson-P1}
  \bD^b(\PP^1) = \langle \cO_{\PP^1}(t-1),\cO_{\PP^1}(t)\rangle,
  \qquad \mbox{for any $t\in \Z$}.
\end{equation}

The right adjoint of $\pi^*$ is $\bR \pi_*$. Let us denote by $\Theta
: \bD^b(\PP^1) \to \bD^b(X)$ the functor that sends $\cF$ to
$\pi^*(\cF) \ts \cO_X(-H)$ and by 
 $\Theta^*$ its left adjoint.
By \eqref{orlov}, any object $\cE$ of $\bD^b(X)$ fits into a
functorial distinguished triangle:
\begin{equation}
  \label{esatto!s}
\pi^* \bR\pi_*\cE \to \cE \to \Theta \Theta^*\cE.
\end{equation}

To compute the expression of these functors, 
we first use \eqref{beilinson-P1} with $t=0$ to check that $\pi^* \bR \pi_*(\cE)$ fits into a functorial distinguished triangle:
\begin{equation}
  \label{Psi}
\HH^\bullet(X,\cE(-F)) \ts \cO_X(-F) \xr{\alpha} \HH^\bullet(X,\cE)
\ts\cO_X \to \pi^* \bR \pi_*\cE, 
\end{equation}
i.e., $\pi^* \bR \pi_*\cE$ is the cone of $\alpha$.

Computing $\Theta^*(\cE)$, cf. \cite[\S 3]{huybrechts:fourier-mukai}
we get:
\[
\Theta^*\cE= \bR \pi_*(\cE(d_XF-H))[1].
\]
Then, using
\eqref{beilinson-P1} with $t=1-d_X$ we get the distinguished triangle: 
\begin{equation}
\label{Thetas}
\HH^\bullet(X,\cE(-H)) \ts \cL(-F) \to
\HH^\bullet(X,\cE(F-H)) \ts\cL \to \Theta
\Theta^*\cE[-1].
\end{equation}

\subsubsection{Exceptional collections and braid group action}

\label{braid-definition}

Let $\epsilon \ge 0$ and set $X = \FF_\epsilon$. A collection of
objects $(\cS_0,\ldots,\cS_r)$ in $\bD^b(X)$ is {\it exceptional}
if it consists of exceptional objects such that $\RHom_X(\cS_j,\cS_k)=0$
if $j>k$. Such collection is {\it full} if it generates $\bD^b(X)$.
Any full exceptional collection on $X$ has $r=3$. As consequence of
Orlov's theorem recalled above, one
of them is:
\begin{equation}
  \label{C1}
\bC_\emptyset=(\cL(-F)[-1],\cL[-1],\cO_X(-F),\cO_X).  
\end{equation}
A motivation for this notation will only be apparent in \S \ref{rigid construction}.

Let us describe the action of the braid group $B_4$ in $4$ strands,
which we number  from $0$ to $3$, on the set of
exceptional collections. Let
$\sigma_i$ be the generator of $B_4$
corresponding to the crossing of the $i$-th strand above the
$(i+1)$-st one. With $\sigma_i$ we associate $\rR_{\cS_{i+1}}\cS_{i}$
so that
$\sigma_i$ sends an exceptional collection $\bC=(\cS_0,\ldots,\cS_3)$
to a new collection $\sigma_i\bC$ where we replace $(\cS_i,\cS_{i+1})$
with $(\cS_{i+1},\rR_{\cS_{i+1}}\cS_{i})$.
Replacing $(\cS_i,\cS_{i+1})$ with $(\rL_{\cS_i}\cS_{i+1},\cS_i)$, 
gives $\sigma_i^{-1}$.
 This satisfies the braid group relations.
The subgroups $B_4$ of braids not involving a given strand
operate on partial exceptional collections.

\medskip

Assume now $X=\FF_\epsilon$ is a del Pezzo surface, i.e. if $\epsilon
\in \{0,1\}$.
Then it
turns out (although we will not need this) that this action is
transitive on the set of all full exceptional collections, cf. \cite[Theorem 6.1.1]{gorodentsev-kuleshov}.
Also, in this case the objects $\cS_i$ are sheaves up to a shift, and actually
(shifted) vector bundles 
if torsion-free. Finally, in this case for $j<k$ there is at most one $i$ such that $\Ext^i_X(|\cS_j|,|\cS_k|) \ne 0$
and $i \in \{0,1\}$, cf. \cite[Proposition 5.3.5]{gorodentsev-kuleshov}.
For any $\epsilon$, an exceptional pair of shifted vector bundles  $(\cS_j,\cS_k)$ on
$\FF_\epsilon$ is called {\it regular} if
$i=0$ and {\it irregular} if $i=1$.

Given an irregular exceptional pair  $(\cP,\cN)$ on $X=\FF_\epsilon$
(for any $\epsilon \ge 0$),
we set $\cG_0=|\cP|[-1]$ and $\cG_1=|\cN|$ and define:
  \begin{align*}
      \cG_{k+1} & = \rR_{\cG_k} \cG_{k-1} && \mbox{for $k \ge 1$}, \\
      \cG_{k-1} & = \rL_{\cG_k} \cG_{k+1} && \mbox{for $k \le 0$}.
  \end{align*}
Note that the two relations we have just written are
both formally valid for any $k \in \Z$.
It turns out that $\cG_{k}$ is concentrated in degree $1$ for $k \le
0$, and in degree $0$ for $k \ge 1$.

Explicitly, we set $w=\hom_X(\cG_0,\cG_1)$ and suppose $w \ge 2$.
For $k \le 0$ we have the sequences: 
\begin{align*}
  & 0 \to \cG_{-2} \to \cG_{-1}^{w} \to \cG_0 \to 0, \\
  & 0 \to \cG_{-3} \to \cG_{-2}^{w} \to \cG_{-1} \to 0, \\
  & 0 \to \cG_{-4} \to \cG_{-3}^{w} \to \cG_{-2} \to 0,
\intertext{etc., while for positive $k$ the first mutation sequences read:}
  & 0 \to \cG_{1} \to \cG_{2}^{w} \to \cG_3 \to 0, \\
  & 0 \to \cG_{2} \to \cG_{3}^{w} \to \cG_{4} \to 0.
\intertext{For the values around $0$ the sequences take the special form:}
  & 0 \to \cG_{1}[-1] \to \cG_{-1} \to \cG_0^{w} \to 0, \\
  & 0 \to \cG_{1}^w \to \cG_{2} \to \cG_{0}[1] \to 0.
\end{align*}

\medskip
Let us recall one more feature of the exceptional objects $\cG_k$
generated by an irregular exceptional pair $(\cP,\cN)$, related to
generalized Fibonacci numbers.
Set $w = \ext^1_X(|\cP|,|\cN|)$ and define the integers $\phi_{w,k}$
recursively as:
\[
\phi_{w,0} = 0, \quad \phi_{w,1} = 1, \quad \phi_{w,k+1} = w
\phi_{w,k}-\phi_{w,k-1}, \quad \mbox {for $k\ge 1$}.
\]
Then, for all $k$, there are exact sequences of sheaves:
\begin{align}
\label{NP>} & 0 \to |\cN|^{\phi_{w,k}} \to \cG_k \to |\cP|^{\phi_{w,k-1}} \to 0, &&\mbox{for $k \ge 1$},\\
\label{NP<} & 0 \to |\cN|^{\phi_{w,-k}} \to \cG_k[1] \to |\cP|^{\phi_{w,1-k}} \to 0,&&\mbox{for $k \le 0$}.
\end{align}
The existence of these sequences (which we sometimes call {\it
  Fibonacci sequences}) is easily carried over to our case from \cite{brambilla:fibonacci}. 
The first Fibonacci sequences look like:
\begin{align*}
  & 0 \to \cG_{1}^{\phi_{w,2}} \to \cG_{2} \to \cG_0[1]^{\phi_{w,1}} \to 0, &
  & 0 \to \cG_{1}^{\phi_{w,3}} \to \cG_{3} \to \cG_{0}[1]^{\phi_{w,2}} \to 0,\\
  & 0 \to \cG_{1}^{\phi_{w,1}} \to \cG_{-1}[1] \to \cG_0[1]^{\phi_{w,2}} \to 0, &
  & 0 \to \cG_{1}^{\phi_{w,2}} \to \cG_{-2}[1] \to \cG_{0}[1]^{\phi_{w,3}} \to 0.
\end{align*}

\subsection{ACM line bundles on scrolls}

ACM line bundles on Hirzebruch surfaces, and more generally on
rational normal scrolls are well-known, cf. \cite{miro-roig:scrolls}.

\begin{lem} \label{linebundles}
Let $\cE$ be an initialized ACM line bundle on $S(\vartheta,\vartheta+\epsilon)$.
Then $\cE \simeq \cO_X(\ell F)$ for $0 \le \ell \le d_X-1$,
or $\cE \simeq \cO_X(H-F)$. Also, $\cE$ is Ulrich iff
$\cE\simeq \cO_X(H-F)$ or $\cE \simeq \cO_X((d_X-1)F)=\cL(H)$.
\end{lem}

\subsection{Basic form of ACM bundles on surface scrolls} \label{ab}

Let now $\cE$ be an ACM sheaf on $X$.
We will compute in more detail the resolution
\eqref{esatto!s}.
Set:
  \[
  a_{i,j} = \hh^i(X,\cE(-j F)), \qquad b_{i,j} =
  \hh^i(X,\cE((1-j)F-H)).
  \]
  We also set $a = a_{1,1}$ and $b = b_{1,0}$.  Since $\cE$ is ACM we have $a_{1,0}=b_{1,1}=0$.
  Then, \eqref{Psi} can be broken up into two exact sequences: 
  \begin{align}
    & \label{f0s} 0\to  \cO_X(-F)^{a_{0,1}}  \to \cO_X^{a_{0,0}} \to \pi^* \pi_*\cE \to \cO_X(-F)^{a}  \to 0, \\
    & \label{per Qs}  0 \to \pi^* \bR^1 \pi_*\cE \to
    \cO_X(-F)^{a_{2,1}} \to \cO_X^{a_{2,0}} \to 0.
  \end{align}
  Since $\Theta \Theta^*\cE$ is the cone of \eqref{Thetas}, we read
  its cohomology in the exact sequences:
  \begin{align}
    &\label{g0}  0 \to \cL^{b} \to \cH^0\Theta \Theta^*\cE \to \cL(-F)^{b_{2,1}} \to  \cL^{b_{2,0}} \to 0,\\
    &\label{dominatess}  0 \to \cL(-F)^{b_{0,1}} \to \cL^{b_{0,0}} \to \cH^{-1}\Theta \Theta^*\cE \to 0.    
  \end{align}
  Taking cohomology of \eqref{esatto!s} we get
   the long exact sequence:
  \begin{equation}
    \label{5pezzis}
    0 \to \cH^{-1} \Theta \Theta^*\cE \xr{\varphi} \pi^*  \pi_*\cE \to
    \cE \to \cH^0\Theta \Theta^*\cE   \to \pi^* \bR^1 \pi_*\cE \to 0.
    \end{equation}

    Let $I$ and $J$ be the images of the middle maps of \eqref{f0s}
    and \eqref{g0}, so that:
    \begin{align}
      \label{I} &    0\to \cO_X(-F)^{a_{0,1}} \to \cO_X^{a_{0,0}} \to      I\to 0, \\
      \label{J} &    0 \to J \xr{g_1} \cL(-F)^{b_{2,1}} \xr{g_2} \cL^{b_{2,0}} \to  0. 
    \end{align}
    We claim:
    \[\pi^* \pi_*\cE = I \oplus \cO_X(-F)^{a}, \qquad
    \cH^0\Theta \Theta^*\cE=J \oplus \cL^{b}.\] 
    Indeed,
    looking at \eqref{f0s}, we have to check that $\pi^* \pi_*\cE \to \cO_X(-F)^{a}$
    splits, and 
    it suffices to prove that $\Ext^1_X(\cO_X(-F),I)=0$.
    This in turn is obtained twisting \eqref{I} by $\cO_X(F)$ and computing cohomology.
    In a similar way one proves that $\cL^{b} \to \cH^0\Theta \Theta^*\cE$ splits.

  Now observe that the restriction of $\pi^* \pi_*\cE \to \cE$ to the
  summand $\cO_X(-F)^{a}$ of 
  $\pi^* \pi_*\cE$ is injective. Indeed, its kernel is $\cH^{-1}\Theta
  \Theta^*\cE$, which is
  dominated by the bundle $\cL^{b_{0,0}}$ in view of
  \eqref{dominatess}. But there are no non-trivial maps $\cL \to
  \cO_X(-F)$.
  Similarly, $\cE \to \cH^0\Theta \Theta^*\cE$ is surjective onto $\cL^{b}$. 
  We define thus the sheaves $P$ and $Q$ by the sequences:
  \begin{align}
    \label{Qs} & 0 \to Q \to J \to \pi^* \bR^1 \pi_*\cE \to 0,\\
    \label{Ps} & 0 \to \cH^{-1} \Theta \Theta^*\cE \to I \to P \to 0.
  \end{align}

  Summing up, we have shown the following basic result.

  \begin{lem} \label{lemmetto}
    Let $\cE$ be an ACM bundle on $X$. Then $\cE$ fits into:
    \[
    0 \to P \oplus \cO_X(-F)^{a} \to \cE \to Q
    \oplus \cL^{b}\to 0, 
    \]
    where $P$ and $Q$ fit into  \eqref{Ps} and \eqref{Qs}, and $I$ and
    $J$ fit into \eqref{I} and \eqref{J}.
  \end{lem}

\subsection{A splitting criterion for scrolls of low degree}

\label{section splitting}

Let us now assume, for the rest of the section, $\epsilon + \vartheta
\le 3$.
Namely, we focus on the
following scrolls (we display their representation type,
although tameness of quartic scrolls will be apparent from \S \ref{quartic}).
\[
\begin{tabular}[h]{c|c|c}
  finite & tame & wild \\
  \hline
  \hline
  $S(1,1)$ &   $S(2,2)$ &  $S(2,3)$  \\
  \hline
  $S(1,2)$ &   $S(1,3)$ &   $S(3,3)$
\end{tabular}  
\]

\begin{lem} \label{s'annullas}
  Let $\cE$ be ACM on $X$, and $P$ and $Q$ as in Lemma \ref{lemmetto}.
  Then $\Ext^1_X(Q,P)=0$.
\end{lem}
\begin{proof}
  To show this, we apply $\Hom_X(Q,-)$ to \eqref{Ps}, obtaining:
  \begin{align*}
    \Ext^1_X(Q,I) \to \Ext^1_X(Q,P) \to \Ext^2_X(Q,\cH^{-1} \Theta \Theta^*\cE).
  \end{align*}
  We want to show that the outer terms of this exact sequence
  vanish.
  For the leftmost term, using \eqref{I}, we are reduced to
  show:
  \begin{align}
    \label{ammazzalo 3s} &\Ext^1_X(Q,\cO_X) \simeq \HH^1(X,Q \ts \cL(-F-H))^*=0, \\
    \label{ammazzalos} &\Ext^2_X(Q,\cO_X(-F)) \simeq \HH^0(X,Q \ts \cL(-H))^*=0,
    \intertext{where the isomorphisms are given by Serre duality.
      For the rightmost term, by \eqref{dominatess} we need:}
    \label{ammazzalo 2s} &\Ext^2_X(Q,\cL) \simeq \HH^0(X,Q(-H-F))^*=0.
  \end{align}
  In turn, by \eqref{Qs},   it suffices to show:
    \begin{align}
      \nonumber & \HH^0(X,J \ts \cL(-H)) = \HH^0(X,J(-H-F))=0, &&
      \mbox{for \eqref{ammazzalos}, \eqref{ammazzalo 2s},} \\
      \label{mo basta} 
      & \HH^1(X,J \ts \cL(-F-H)) = \HH^0(X,\pi^*
      \bR^1 \pi_*\cE \ts \cL(-F-H)) =0, && \mbox{for \eqref{ammazzalo 3s}}.
    \end{align}
    The first line follows by taking global sections
    of \eqref{J}, twisted by $\cL(-H)$, or by $\cO_X(-H-F)$.
    Similarly, the first vanishing required for \eqref{mo basta}
    follows from \eqref{J} and Serre duality, since $\HH^2(X,\cL^*)=0$
    and $\HH^1(X,\cL^*(F))=0$
    for $\epsilon + \vartheta \le 3$ (here
    is where the bound on the invariants appears).
    The last vanishing  follows
    taking global sections of \eqref{per Qs}, twisted by $\cL(-F-H)$.
    \end{proof}

  \begin{prop} \label{criterions}
    Any non-zero ACM bundle $\cE$ satisfying:
    \begin{equation}
      \label{thensplitss}
      \HH^1(X,\cE(tH-F)) = 0, \quad \HH^1(X,\cE(tH+F))=0, \quad
      \mbox{for all $t \in  \Z$,} 
  \end{equation}
     splits as a direct
    sum of line bundles.
  \end{prop}
  \begin{proof}
    We borrow the notation from  the above discussion,
    and we assume $a=b=0$.
    First note that the sheaves $I$ and $\cH^{-1} \Theta \Theta^*\cE$ are
    torsion-free and in fact locally free, since they are obtained as
    pull-back of direct images of $\cE$ and 
    $\cE(-H)$ via $\pi$.
    More precisely, by \eqref{I}, there are integers $r \ge 1$
    and $p_i \ge 0$ such that:
    \begin{equation}
      \label{Isplits}
    I \simeq \bigoplus_{i=0,\ldots, r}\cO_X(iF)^{p_i}.
    \end{equation}

    We can assume that $\cE$ is initialized, hence
    $\HH^0(X,\cE(-H))=0$.
    We get:
    \begin{align}
    \label{-1,1s} &\cH^{-1} \Theta \Theta^*\cE \simeq \cL^{b_{0,0}}.
    \end{align}

    In view of the previous lemma, we have that $\cE$ is the direct
    sum of $P$ and $Q$.
    Since $\HH^0(X,Q)=0$, to conclude, it remains to prove that $P$ is
    a (possibly zero) direct sum of line 
    bundles. Indeed, if $P \ne 0$, we may split off $P$ from $\cE$ and
    use induction on 
    the rank to get out statement; on the other hand $P = 0$ leads to a
    contradiction since $\HH^0(X,\cE=\ne 0$ and $\HH^0(X,Q)=0$.
    Using \eqref{-1,1s}, the exact sequence \eqref{Ps} becomes:
    \[
    0 \to \cL^{b_{0,0}} \to \bigoplus_{i=0,\ldots, r}\cO_X(iF)^{p_i} \to    P \to 0.
    \]
    Twisting this sequence by $\cO_X(-H-F)$, since
    $\HH^k(X,\cO_X((i-1)F-H))=0$ for any $i$ and any $k$, we get 
    $\HH^1(X,P(-H-F)) \simeq \bk^{b_{0,0}}$.
    But this space must be zero, since $P$ is a direct
    summand of $\cE$, and $\cE$ satisfies \eqref{thensplitss}.
    We deduce
    $P \simeq \bigoplus_{i=1,\ldots, r}\cO_X(iF)^{p_i}$.
  \end{proof}

This allows to give the following refinement of Lemma \ref{lemmetto}.

\begin{lem} \label{soloO}
  Let $\cE$ be an indecomposable ACM bundle on $X$, assume $\vartheta
  \ge 2$ and let
  $I$, $J$, $P$, $Q$, $a_{i,j}$ and $b_{i,j}$ be as above.
  If $P \ne 0$ and $b \ne 0$, then $\vartheta + \epsilon = 3$ and $P \simeq I \simeq \cO_X^{a_{0,0}}$.
\end{lem}

\begin{proof}
As in the proof of Proposition \ref{criterions}, $I$ takes the form
 \eqref{Isplits}. Set $p=p_0$ and rewrite $I$ as:
\[
I = \cO_X^{p} \oplus I', \qquad I'= \bigoplus_{i=1,\ldots,r}\cO_X(iF)^{p_i}.
\]
This time, by \eqref{dominatess} we get a
splitting of
$\cH^{-1} \Theta \Theta^* \cE$ of the form:
\begin{equation}
  \label{hmenouno}
\cH^{-1} \Theta \Theta^* \cE \simeq \bigoplus_{i=1,\ldots, s}\cL(b_iF),
\end{equation}
for some integers $s \ge 0$, $b_i \ge 0$.

We now observe that the map $\varphi$ appearing in \eqref{5pezzis} is
zero, when restricted to the summand $\cO_X^p$ of $I$, because
$\HH^0(X,\cL^*)=0$ when $\vartheta \ge 2$.
Therefore, $P = \cO_X^p \oplus P'$, where $P'$ fits
into:
\begin{equation}
  \label{eq:I'}
0 \to \cH^{-1} \Theta \Theta^* \cE \to I' \to P' \to 0.  
\end{equation}

Next, we use that $\cE$ is indecomposable. 
We tensor the previous exact sequence with $\cL^*$ and note that $i
\ge 1$ implies $\HH^1(X,I' \ts \cL^*))=0$, while \eqref{hmenouno} easily
gives $\HH^2(X,\cH^{-1} \Theta \Theta^* \cE \ts \cL^*)=0$.
Therefore $\HH^1(X,P'\ts \cL^*)=0$. 
In view of Lemma
\ref{s'annullas}, since $P = P' \oplus \cO_X^p$ and $\cE$ is
indecomposable, this implies $P'=0$ and $\HH^1(X,\cL^*) \ne 0$. This says
$\vartheta + \epsilon = 3$.
Now \eqref{eq:I'} implies that $\cH^{-1} \Theta \Theta^* \cE$ and $I'$
are both zero (indeed, they should be isomorphic, but they are direct sums
of line bundles with different
coefficients of $H$), so $P \simeq I \simeq \cO_X^{a_{0,0}}$.
\end{proof}

An essentially identical argument shows that, if $\cE$ is
indecomposable ACM with $Q
\ne 0$ and $a \ne 0$, then $Q \simeq J \simeq \cL(-F)^{b_{2,1}}$
and $\vartheta + \epsilon = 3$.
We have proved:

\begin{prop} \label{prop-abcd}
  Let $\epsilon + \vartheta \le 3$, $\vartheta \ge 2$ and $\cE$ be an
  indecomposable ACM bundle on $X$. Set:
\[
a=a_{1,1}, \quad b=b_{1,0}, \quad c=a_{0,0}, \quad d=b_{2,1}.  
\]
Then $\cE$ fits into:
\[
0 \to \cO_X^\cc \oplus \cO_X(-F)^\aaa \to \cE \to \cL^\bb
\oplus \cL(-F)^\dd \to 0.
\]
\end{prop}

\subsection{Monads}

\label{monadology}

Let $\epsilon + \vartheta \le 3$ and $\vartheta \ge 2$. Given an indecomposable ACM bundle $\cE$ on $X=S(\vartheta,\vartheta+\epsilon)$,
by Proposition \ref{prop-abcd} we can consider
the kernel $\cK$ of the natural projection $\cE \to \cL(-F)^{d}$,
and the cokernel $\cC$  of the natural injection $\cO_X^{c} \to
\cE$.
This injection factors through $\cK$ and we let
$\cF$ be cokernel of the resulting map. We will see in a minute that $\cF$ is an Ulrich bundle.
We have thus a complex whose cohomology is $\cF$ (a {\it monad})
of the form:
\[
0 \to \cO_X^{c} \to \cE \to \cL(-F)^{d} \to 0.
\]
The {\it display} of the monad is the following commutative exact diagram:
\begin{equation}
  \label{diagram-monad}
  \xymatrix@R-3.5ex@C-3ex{& 0 \ar[d] & 0 \ar[d]\\
    & \cO_X^\cc \ar@{=}[r] \ar[d] & \cO_X^\cc \ar[d]  \\
    0 \ar[r] & \cK \ar[r]  \ar[d]  & \cE \ar[r] \ar[d] & \cL(-F)^\dd \ar@{=}[d]\ar[r] & 0\\
    0 \ar[r] & \cF \ar[r] \ar[d] & \cC \ar[r] \ar[d] & \cL(-F)^\dd \ar[r] & 0\\
    & 0 & 0 }
\end{equation}

Note that some monads  can also be constructed when $\vartheta=1$ 
cf.  the following example.

\begin{eg} \label{V-W}
  Let $X=S(1,3)$, and observe that $\hh^1(X,\cL^*)=1$. Define the rank-$2$
  bundle $\cV$ as the non-trivial extension:
  \[
  0 \to \cO_X  \to \cV \to \cL \to 0.
  \]
  So in this case $c=b=1$ and $a=d=0$.
  Note that $\cV$ is slope-unstable, and of course ACM.
  It is clear that $\HH^1(X,\cV^*)=0$, since $1 \in \HH^0(X,\cO_X)$ is
  sent by the boundary map to the generator of $\HH^1(X,\cL^*)$.
  Also,  $\cV \ts \cV^*$ fits into:
  \[
  0 \to \cV^* \to \cV \ts \cV^* \to \cV \to 0,
  \]
  and it easily follows that $\Ext^1_X(\cV,\cV)=0$, so $\cV$ is rigid
  (albeit not simple, since $\hom_X(\cV,\cV)=2$).
  Taking $\cV(-F)$, we get a monad with $a=d=1$ and $b=c=0$.

  However, for $X=S(1,3)$, Proposition \ref{prop-abcd} does not always
  apply. Indeed, observe:
  \[
  \ext^1_X(\cL,\cO_X(H-F))=\hh^1(X,\cO_X(2H-4F))=1,
  \]
  and define the bundle $\cW$ as fitting in the associated
  non-trivial extension:
  \[
  0 \to \cO_X(H-F) \to \cW \to \cL \to 0.
  \]
  Clearly $\cW$ is an initialized ACM sheaf on $X$. Also, the displayed
  extension is Harder-Narasimhan filtration of $\cW$, which shows that
  $\cW$ is indecomposable. Just as for $\cV$, one checks that $\cW$ is rigid.
\end{eg}

\section{Parametrizing Ulrich bundles via the Kronecker quiver}

\label{section kronecker}

Let again $X$ be the scroll $S(\vartheta,\vartheta+\epsilon)$  
of degree $d_X = 2 \vartheta + \epsilon$.
We will carry out here a description of Ulrich bundles on $X$ as
extensions, and relate them to representations of a Kronecker quiver.

\subsection{Ulrich bundles as extensions}
The first consequence of the computation we carried out in the previous
section is the next result, which proves part \eqref{parte_estensioni} of Theorem
\ref{main-wild}.

  \begin{prop} \label{ulrichprop}
    A vector bundle $\cE$ on $X$ is Ulrich iff,
    up to a twist, it fits into:
    \begin{equation}
      \label{ulrich-extension}
    0 \to \cO_X(-F)^a \to \cE \to \cL^b \to 0,
    \end{equation}
    with $a= \hh^1(X,\cE(-F))$ and $b = \hh^1(X,\cE(F-H))$.
  \end{prop} 

\begin{proof}
Let us borrow notations from \S \ref{ab}.
By Lemma \ref{linebundles}, if $\cE$ fits into
\eqref{ulrich-extension} for some $a$ and $b$, then it is obviously
Ulrich.
Also, in this case the integers $a$ and $b$ equal $a_{1,1}$ and,
respectively, $b_{1,0}$ as one sees by tensoring
\eqref{ulrich-extension} with $\cO_X(-F)$ and $\cO_X(F-H)$ and taking
cohomology.

Conversely, if $\cE$ is Ulrich,
then (up to a twist) we may assume
$a_{0,0}=b_{2,1}=0$ by \cite[Proposition 2.1]{eisenbud-schreyer-weyman}.
Therefore, by \eqref{I} and \eqref{J} we get $I=J=0$ hence
$P=Q=0$ by \eqref{Ps} and \eqref{Qs}.
We conclude by Lemma \ref{lemmetto}.
\end{proof}

\begin{rmk} \label{occhioallerrore}
  We have $\ext^1_X(\cL,\cO_X(-F))>2$ for all $d_X > 4$.
  Then, extensions of the form \eqref{ulrich-extension} provide
  arbitrarily large families of Ulrich bundles, except for
  $d_X \le 4$. This is the argument of \cite{miro-roig:scrolls}, which
  overlooks the case $d_X=4$ only.
\end{rmk}

Recall that Ulrich bundles are semistable. Then, their
moduli space is a subset of Maruyama's moduli space of $S$-classes
of semistable sheaves (with respect to $H$) of fixed Hilbert
polynomial.

\begin{corol}
  For any integer $r \ge 1$, the moduli space of Ulrich bundles of
  rank $r$ is supported at $r+1$ distinct points, characterized by $c_1(\cE)$.
\end{corol}

\begin{proof}
  If $\cE$ is an Ulrich bundle of rank $r$, then by 
  Proposition \ref{ulrichprop} $\cE$ fits into an exact sequence of the form
  \eqref{ulrich-extension}, with $a_{1,1}+b_{1,0}=r$.
  Set $a=a_{1,1}$, so $b_{0,1}=r-a$.
  Since $\cO_X(-F)$ and $\cL$ have the same Hilbert
  polynomial, the graded object associated with $\cE$ is 
   $\cO_X(-F)^{a} \oplus \cL^{r-a}$.
  Hence the
  $S$-equivalence class of $\cE$ only depends on $a$, and in turn $a$
  is clearly determined by $c_1(\cE)$.
\end{proof}

\subsection{Kronecker quivers}

Let $w \ge 2$ and let $\Upsilon_w$ be the Kronecker quiver
with two vertexes
$\mathbf{e_1}$, $\mathbf{e_2}$ and $w$
arrows from $\mathbf{e_1}$ to $\mathbf{e_2}$.
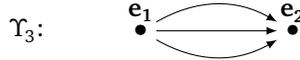
\[\begin{tikzpicture}[scale=1]
  \draw (-2.5,0) node [] {$\Upsilon_3$:}; 
  \draw (-1,0) node [above] {$ \mathbf{e_1}$};
        \draw (1,0) node [above] {$\mathbf{e_2}$};
        \node (1) at (-1,0) {$\bullet$};
        \node (1) at (-1,0) {$\bullet$};
        \node (2) at (1,0) {$\bullet$};
        \draw[->,>=latex] (1) to (2);
        \draw[->,>=latex] (1) to[bend left] (2);
        \draw[->,>=latex] (1) to[bend right] (2);
    \end{tikzpicture}
\]

\subsubsection{Representations of Kronecker quivers}
Given integers $a$ and $b$, a representation $\cR$ of $\Upsilon_X$ of dimension vector $(b,a)$ is given
by a pair of vector spaces, $B$ and $A$ with $\dim(B)=b$ and
$\dim(A)=a$ and $w$ linear maps $B^* \to A$. We have the obvious
notions such as direct sum, irreducible representation etc.
If $\cR$ has dimension vector $(b,a)$, its deformation space has virtual dimension:
\[
\psi(a,b)= w ab-a^2-b^2+1.
\]

We think of a representation of $\Upsilon_w$ of dimension vector
$(b,a)$ as a matrix $M$ of size $a \times b$ whose entries are
linear forms in $w$ variables, namely if our $w$ linear maps have
matrices $M_1,\ldots,M_w$ in a fixed basis, we write
$M=x_1M_1+\cdots+x_w M_w$. If $w=2$ we write $\sx=x_1$ and $\sy=x_2$.

Let $W$ be a vector space of dimension $w$ with a basis indexed
by the arrows of $\Upsilon_w$.
The representation $\cR$ then
corresponds uniquely to an element $\xi$ of $A \ts B \ts W^*$, so we write
 $\cR = \cR_\xi$.
 We write $M_\xi$ the matrix corresponding to $\xi
\in A \ts B \ts W^*$. It is customary to write $M_\xi=M_{\xi'}\bp
M_{\xi''}$ when $\cR_\xi=\cR_{\xi'} \oplus \cR_{\xi''}$.
Geometrically, 
the space $\PP W$ parametrizes isomorphism classes of non-zero
representations of $\Upsilon_w$ with dimension vector $(1,1)$.
The matrix $M_\xi$ is naturally written as a morphism of sheaves over the dual space:
\[
M : B^*\ts \cO_{\PD W}(-1) \to A\ts \cO_{\PD W}.
\]
Via the natural isomorphism $H^0(\PD W,\cO_{\PD W}(1)) \simeq W^*
\simeq H^0(\PP W,\cT_{\PP W}(-1))$, the matrix $M_\xi$ is transformed into a
matrix of twisted vector fields.

We write $\bD^b(\Upsilon_w)$ for the derived category of
$\bk$-representations of $\Upsilon_w$.
It is naturally equivalent to
the full subcategory $\langle \Omega_{\PP  W}(1), \cO_{\PP W}\rangle$ of
$\bD^b(\PP W)$.
Given a non-zero vector of $W^*$ and the
corresponding point $\sa \in \PP W$, there is a morphism $\Omega_{\PP
  W}(1) \to \cO_{\PP W}$ vanishing at $\sa$.
The cone $S_\sa$ of this morphism is mapped by this equivalence to the associated
$(1,1)$-representation $\cR_\sa$ of $\Upsilon_w$.

\subsubsection{Irregular exceptional pairs and the universal extension} 
\label{irreg}

Let $(\cP,\cN)$ be an exceptional pair over a surface scroll $X$, 
with $\cP=|\cP|$ and $\cN=|\cN|$, and $\Ext^i_X(\cP,\cN)=0$ for $i\ne 1$.
Define $W^* = \Ext^1_X(\cP,\cN)$ and $w=\dim(W)$.
Over $X \times \PP W$ we have the universal extension:
\[
   0 \to \cN \boxtimes \cO_{\PP W} \to \cW \to \cP \boxtimes \cO_{\PP
  W}(-1) \to 0.
\]
The functor $\Phi : \bD^b(\PP W) \to \bD^b(X)$ defined by $\bR
q_*(p^*(-) \ts \cW)$ gives an equivalence of
$\bD^b(\Upsilon_w)$ onto the subcategory $\langle \cP ,\cN\rangle \subset \bD^b(X)$
sending $\cR_\xi$ to the extension $\cF=\cW|_{X \times \{\xi\}}$ fitting into:
\[
0 \to A \ts \cN \to \cF \to B^* \ts \cP \to 0.
\]
Write $\cF_\sa$ for the extension of $\cP$ by $\cN$ corresponding to
$\sa \in \PP W$. We have: 
\begin{equation}
  \label{F2 e F1}
\Phi(\cO_{\PP W}) \simeq \cN, \qquad
\Phi(\Omega_{\PP W}(1)) \simeq \cP[-1], \qquad \Phi(S_\sa) \simeq \cF_\sa.
\end{equation}

\begin{lem} \label{aspita} Let $w\ge 2$, $\xi \in A \ts B \ts W^*$ and set $\cR=\cR_\xi$ and $\cF=\Phi(\cR)$.
  \begin{enumerate}[i)]
  \item \label{ff} We have $\Ext^i_{\Upsilon_w}(\cR,\cR) \simeq
    \Ext^i_{X}(\cF,\cF)$, and this space is zero for $i \ne 0,1$.
  \item \label{ind} The bundle
 $\cF$ is indecomposable if and only if $\cR$ is
    irreducible.
  \item  \label{psicasi} Let $(b,a)$ be the  dimension vector of $\cR$. Then we have
    the cases.
    \begin{enumerate}
    \item If $\psi(a,b)<0$ then $\cR$ is decomposable, hence so is $\cF$.
    \item If $\cF$ is
      indecomposable and rigid then  $\psi(a,b)= 0$, and this happens if and only if
       $\{a,b\} = \{\phi_{w,k},\phi_{w,k+1}\}$ for some $k \ge 0$, cf. \S \ref{braid-definition}.
    \item \label{citoqua} If $\psi(a,b)=0$, then $\cF$ is exceptional for general
      $\xi$. Also, $\cF$ is exceptional if and only if there is a
      well-determined integer $k$ such that $\cF \simeq
      \cG_k$ and $k \ge 1$ or $\cF \simeq\cG_k[1]$ and $k \le 0$.
    \end{enumerate}
  \item \label{somma} The bundle $\cF$ is rigid if and only if there are integers
    $k \ne 0$, $a_k$ and $a_{k+1}$ such that:
    \[
    \cF \simeq |\cG_{k}|^{a_{k}}  \oplus |\cG_{k+1}|^{a_{k+1}}.
    \]
  \end{enumerate}
\end{lem}

\begin{proof}
   \eqref{ff} is clear since $\Phi$ is fully faithful and
  representations of $\Upsilon_w$ form a hereditary category.
  For \eqref{ind}, note that $\cF$ is decomposable if and only if
  $\Hom_{\Upsilon_w}(\cR,\cR)$  contains a non-trivial idempotent. By
  \eqref{ff}, this happens if and only if $\Hom_{\Upsilon_w}(\cR,\cR)$ contains a
  non-trivial idempotent, i.e., if and only if $\cR$ is irreducible.
  
  Part \eqref{psicasi} follows from \cite[Theorem
  4]{kac}, cf. also \cite[Remarks a, b]{kac}, as $\psi(a,b)<0$
  directly implies that $\cR$ is decomposable, while $\psi(a,b)=0$
  precisely means that  $(b,a)$ is a Schur root in the sense of Kac.
  Since indecomposable Schurian representations are well-known to be
  uniquely determined by their dimension vector, we conclude that 
  they correspond precisely to the 
  Fibonacci
  bundles $|\cG_k|$ by
  \cite{brambilla:fibonacci, brambilla:simplicity}.
  Part \eqref{somma} follows from \cite[Lemma
  5]{faenzi:iyama-yoshino}.
\end{proof}

\subsubsection{Proof of parts \eqref{primaparte} and \eqref{secondaparte} of Theorem \ref{main-wild}
} \label{ulrich-notation}

Now we apply this approach to Ulrich bundles.
If $\cU$ is an  Ulrich bundle over $X$, then, up to
a twist, $\cU$
fits into an extension of the form \eqref{ulrich-extension} by
Proposition \ref{ulrichprop}.
Set $W = \Ext^1_X(\cL,\cO_X(-F))^*$, so that $w = \dim(W)=d_X-2$.
Then, there is some $\xi \in A
\ts B \ts W^*$ determined up to a non-zero scalar such that $\cU \simeq
\Phi(\cR_\xi)$. We write $\cU = \cU_\xi$, and we call this the 
 Ulrich bundle associated with $\xi$. Note that this is an instance of a bundle
 referred to as $\cF=\Phi(\cR_\xi)$ from \S \ref{irreg}, but we use
 the letter $\cU$ when dealing with Ulrich bundles.

Having this in mind we see that, if $\cU$ is indecomposable and
rigid, then looking at the weight vector $(b,a)$ of $\cR_\xi$,
the unordered pair $\{a,b\}$ must give a Schur root, i.e. $a$ and $b$
must be two
Fibonacci numbers of the form $\phi_{w,k}$, $\phi_{w,k-1}$, for $k
\ge 1$. Moreover for the weight vector $(b,a)=(\phi_{w,k-1},\phi_{w,k})$
there is a unique indecomposable bundle, which we
denote by $\fU_k$, and each $\fU_k$ is exceptional (we will use
gothic letters only for rigid objects).

We notice that,
since $\cO_X(-F)$ and $\cL$ are interchanged by dualizing and twisting
by $\omega_X(H)$, the bundles $\fU_{k}(-H)$ and $\fU_{1-k}^* \ts \omega_X$ are both indecomposable
and rigid with the same weight vector, and must then be isomorphic by
Lemma \ref{aspita}, part \eqref{citoqua}.

As an explicit example, over $X=S(2,3)$, the first Fibonacci sequences for
$\fU_k$ are:
\begin{align*}
  &\fU_1 = \cO_X(-F), & & 0 \to \cO_X(-F)^3 \to \fU_{2} \to \cL \to 0, \\
  &&&0 \to \cO_X(-F)^{8} \to \fU_{3} \to \cL^3 \to 0, \\
  &\fU_0 = \cL, 
  &&0 \to \cO_X(-F) \to \fU_{-1} \to \cL^3 \to 0, \\
  &&&0 \to \cO_X(-F)^3 \to \fU_{-2} \to \cL^8 \to 0, \\
  &&&0 \to \cO_X(-F)^{8} \to \fU_{-3} \to \cL^{21} \to 0.
\end{align*}

\subsection{Matrix pencils}

Representations of Kronecker quivers 
are completely classified only if $w=2$. In this case we canonically
have $W^* \simeq W$. We write $\PP^1=\PP W$, so $\bD^b(\Upsilon_2)
\simeq \bD^b(\PP^1)$, and morphisms of the form $M=M_\xi$ are matrix pencils.
These are classified by
Kronecker-Weierstrass theory, which we recall for the reader's convenience.
We refer to \cite[Chapter
19.1]{burgisser-clausen-shokrollahi} for proofs.
Fixing variables $\sx,\sy$ on $\bP$, and given positive integers
$u$, $v$, $n$ and $\sa \in \bk$ one defines:
\[
\begin{footnotesize}
 \fC_{u} =
 \begin{pmatrix}
   \sx & \\
   \sy & \sx & \\
   &    \sy & \ddots & \\
   && \ddots & \sx \\
   &&&    \sy
 \end{pmatrix},
\hspace{0.4cm}
     \fB_{v} =
 \begin{pmatrix}
   \sx & \sy \\
   &    \sx & \sy \\
   && \ddots & \ddots\\
   &&&    \sx & \sy
 \end{pmatrix},
\hspace{0.4cm}
J_{\sa,n} =
\begin{pmatrix}
  \sa & 1 \\
  & \ddots & \ddots\\
  &&\sa &    1 \\
  &&&    \sa
\end{pmatrix},
\end{footnotesize}
\]
and $\fJ_{\sa,n} = \sx I_n + \sy J_{\sa,n}$,
where $\fC_u$ has size $(u +1)\times u$,
 $\fB_v$ has size $v \times (v +1)$, and $J_{\sa,n} \in \bk^{n \times n}$.
The next lemma is obtained combining \cite[Theorem 19.2 and
19.3]{burgisser-clausen-shokrollahi}, with the caveat that, up to
changing basis in $\bP$, we can assume that a matrix pencil $M$ has
no {\it infinite elementary} divisors, i.e., the morphism $M:
B^* \ts \cO_{\bP}(-1) \to A \ts \cO_{\bP}$ has constant rank
around $\infty=(0:1)  \in \bP$.

\begin{lem} \label{KW}
Up to possibly changing basis in $\bP$, any matrix pencil $M$ is
equivalent to:
\[
\fC_{u_1} \bp \cdots \bp \fC_{u_r} \bp
\fB_{v_1} \bp \cdots \bp \fB_{v_s} \bp \fJ_{n_1,\sa_1} \bp \cdots
\bp \fJ_{n_t,\sa_t} \bp \fZ_{a_0,b_0},
\]
for some integers $r,s,t,a_0,b_0$ and $u_i,v_j,n_k$, and
some $\sa_1,\ldots,\sa_t \in \bk$.
\end{lem}

We have the following straightforward isomorphisms:
\[
\coker(\fC_u) \cong \cO_{\bP}(u), \qquad \ker(\fB_v) \cong
\cO_{\bP}(-v-1), \qquad \coker(\fJ_{n,\sa}) \cong \cO_{n \sa},
\]
where $\cO_{n \sa}$ is the skyscraper sheaf over the point $\sa$
with multiplicity $n$.
Allowing $\sa$ to vary in $\bP$ instead of $\bk$ only amounts to
authorizing infinite elementary divisors too.

\begin{prop} \label{prop-ulrich}
  Assume $w=2$, let $\xi$ be an element of $\notB \ts \notA \ts W$ and set
  $\cE=\cU_\xi$, $M=M_\xi$.
  If $\cE$ is indecomposable, then $\Ext^2_X(\cE,\cE)=0$,
   $|a -b| \le 1$ and:
  \begin{enumerate}[(i)]
  \item   if $a=b+1$ then $M \simeq \fB_b$ and  $\cE$ is
    exceptional;
  \item   if $a=b-1$, then $M \simeq \fC_{b-1}$ and $\cE$ is exceptional;
  \item   if $a=b$ then $M \simeq \fJ_{a,\sa}$ for some $\sa \in \PP^1$, and the
    indecomposable deformations of $\cE$ vary in a $1$-dimensional family, parametrized by a projective line.
  \end{enumerate}
\end{prop}

\begin{proof} 
By Lemma \ref{aspita} we can assume that $M$ is
irreducible, so that $M$ is itself isomorphic to one of the summands appearing in Lemma \ref{KW}.
Consider the equivalence:
\begin{equation}
  \label{3maniere}
 \langle \cP,\cN \rangle \simeq \bD^b(\Upsilon_2) \simeq \bD^b(\PP^1). 
\end{equation}

If $M\simeq \fC_u$, then this equivalence maps  $\cE$ to $\coker(\fC_u) \simeq \cO_{\bP}(u)$, with $u \ge 1$.
Since $\cO_{\bP}(u)$ is exceptional, Using part \eqref{ff} of Lemma \ref{aspita}, we get that $\cE$ is an
also an exceptional bundle.
For the case $\fB_v$, \eqref{3maniere} sends $\cE$ to $\ker(\fB_v)[1] \simeq
\cO_{\bP}(-v-1)[1]$, with $v \ge 1$.
Again $\cE$ is then exceptional.
The same argument works for $\fZ_{1,0}$ in which
case  $\cE \simeq \cL$ is mapped to $\cO_{\bP}(-1)[1]$, 
and for $\fZ_{0,1}$, whose counterpart on $\PP^1$ is $\cO_{\bP}$, and
of course $\cE \simeq \cO_X(-F)$.

It remains to look at the case $M \simeq \fJ_{n,\sa}$, so that $\cE$
is mapped under \eqref{3maniere} to $\cF = \cO_{n
  \sa}$. In this case, $\cF$ is filtered by the sheaves $\cF_m=\cO_{m \sa}$, for
$1 \le m \le n$ and $\cF_m/\cF_{m-1} \simeq \cF_1$.
This induces a filtration of $\cE$ be the sheaves $\cE_m$, image of
$\cF_m$ via \eqref{3maniere},
having quotients $\cE_m/\cE_{m-1} \simeq \cE_{\sa}$, recall \eqref{F2 e F1}.
Note  that $\cE_{\sa}$ is a simple bundle with
$\Ext^1_X(\cE_{\sa},\cE_{\sa}) \simeq \bk$ once more by Lemma \ref{aspita}, and
$\Ext^2_X(\cE_{\sa},\cE_{\sa}) =0$. Bundles of the form $\cE_{\sa}$
are of course parametrized by $\PP^1$.
Deformations of $\cE$ are thus provided by the
motions in $\PP^1$ of each of the factors $\cE_{\sa}$ of its filtration.
But only bundles associated with sheaves of the form $\cO_{n \sa'}$ continue
to be indecomposable, so deformations of $\cE$
giving rise to indecomposable bundles correspond exactly to the
motions of ${n \sa}$ in  $\bP$ itself.
\end{proof}

\section{Quartic scrolls are of tame representation type}

\label{quartic}

Our goal here is to prove Theorem \ref{main-tame}. According to
Bertini and del Pezzo's classification, cf.
\cite{eisenbud-harris:centennial}, a smooth
non-degenerate surface $X$ in $\PP^5$ has degree $4$ if and only if
$X$ is a quartic scroll or a Veronese surface, the last case being
well understood and excluded of our study. So we actually show two
stronger statements, one for each of the two non-isomorphic smooth
quartic scrolls, namely $S(2,2)$ and $S(1,3)$, which we treat
separately in \S \ref{S22} and \S \ref{S13}. The outcome is a complete 
classification of ACM bundles on these scrolls.
Here, continuous families of indecomposable ACM sheaves only exist for Ulrich
bundles of even rank, which are parametrized by $\PP^1$. 
In both cases, all other indecomposable ACM sheaves are rigid, but for
$S(1,3)$ the classification 
of non-Ulrich ACM sheaves is a bit more involved.

\subsection{ACM bundles on $S(2,2)$} \label{S22}

We now show a stronger version of Theorem \ref{main-tame} for $X=S(2,2)$.

\begin{thm} \label{quartic-is-tame}
  Let $X=S(2,2)$ and let $\cE$ be an indecomposable ACM
  bundle on $X$.
  Then, up to a twist, either $\cE$ is $\cO_X$ or $\cL_X(F)$ or $\cL(-F)$, either $\cE$ is Ulrich, and can be
  expressed as an extension:
  \[
  0 \to \cO_X(-F)^a \to \cE \to \cL^b \to 0,
  \]
  for some $a,b \ge 0$ with $|a-b| \le 1$.
  In this case:
  \begin{enumerate}[i)]
  \item  for any $(a,b)$ with $a = b \pm 1$, there exists a
    unique indecomposable bundle of the form $\cE$ as above,
    and moreover this bundle is exceptional;
  \item \label{occhio} for $a = b \ge 1$, the isomorphism classes of indecomposable
    bundles of the form $\cE$ are parametrized by $\bP$.
  \end{enumerate}
\end{thm}

\begin{proof}
  We assume $\cE$ is an indecomposable ACM bundle which
  is not Ulrich over $X=S(2,2)$, and prove that  $\cE$ is then a line
  bundle,   all the remaining statements being clear by Proposition
  \ref{prop-ulrich}.

  In view of Proposition \ref{ulrichprop}, we know
  that replacing $\cE$ with $\cE(t H)$ for any $t \in \Z$, we always get $P \ne 0$ or
  $Q \ne 0$.
  But, if $P \ne 0$, then $b=0$
  by Lemma \ref{soloO} since $\vartheta + \epsilon = \vartheta = 2$, and $Q=0$ by
  Lemma \ref{s'annullas}. By Lemma \ref{lemmetto} we now deduce $a=0$.
  Likewise if we start with $Q \ne 0$ we conclude $a=b=0$. 
  By Proposition \ref{criterions}, we deduce that $\cE$ is a line
  bundle. The conclusion follows.

\end{proof}

\subsection{ACM bundles on $S(1,3)$} \label{S13}

Here we take up the second type of quartic scroll, $X=S(1,3)$. In this
case  a bit more effort is required to classify ACM bundles,
in order to take into account a larger set of sporadic rigid bundles,
essentially $\cV$ and $\cW$ of Example \ref{V-W}.
Our goal will be to prove the following result.

\begin{thm} \label{S(1,3)-is-tame}
  Let $X=S(1,3)$. Then any indecomposable ACM
  bundle $\cE$ on $X$ is either $\cE$ an Ulrich bundle, or a line bundle, or
  a twist of $\cV$, $\cV(-F)$, $\cW$.
  If $\cE$ is Ulrich then it is an extension:
  \[
  0 \to \cO_X(-F)^a \to \cE \to \cL^b \to 0,
  \]
  for some $a,b \ge 0$ with $|a-b| \le 1$.
  So:
  \begin{enumerate}[i)]
  \item \label{ocio} for $(a,b)$ with $a = b \pm 1$, there exists a
    unique indecomposable $\cE$ as above and $\cE$ is exceptional;
  \item \label{occhio2} for $a = b \ge 1$, the isomorphism classes of
    $\cE$'s as above are parametrized by $\bP$.
  \end{enumerate}
\end{thm}

\subsubsection{The monad for $S(1,3)$}

Let $\cE$ be an ACM bundle on $X$. The procedure of \S
\ref{monadology} in this case gives a monad which is tantamount to the
following display:
\begin{equation}
  \label{diagram-monad-2}
  \xymatrix@R-3.5ex@C-3ex{& 0 \ar[d] & 0 \ar[d]\\
    & P\ar@{=}[r] \ar[d] & P \ar[d]  \\
    0 \ar[r] & \cK \ar[r]  \ar[d]  & \cE \ar[r] \ar[d] & Q \ar@{=}[d]\ar[r] & 0\\
    0 \ar[r] & \cU \ar[r] \ar[d] & \cC \ar[r] \ar[d] & Q \ar[r] & 0\\
    & 0 & 0 }
\end{equation}
with $P,Q$ defined in Lemma \ref{lemmetto} and where $\cU$ is an
Ulrich bundle, called the {\it Ulrich part} of $\cE$.

\subsubsection{Two vanishing results for Ulrich bundles and a lemma on rigid extensions}

We first need to prove that Ulrich bundles of rank greater than one do
not mix with other elements of our monad.

\begin{lem} \label{not containing}
  Any Ulrich bundle $\cU$ on $X$, not containing $\cL$
  as a direct summand, satisfies $\HH^1(X,\cU^*)=0$.
  Likewise $\HH^1(X,\cU \ts \cL^*(F)))=0$, unless $\cO_X(-F)$ is a
  summand of $\cU$.
\end{lem}

\begin{proof}
  Recall the notation from \S \ref{hirzebruch} and \ref{ulrich-notation}. Observe that $U = \HH^0(X,\cO_X(F))$ is identified with
  $\HH^1(X,\cL^*(-F))$ and that $W$ and $U$ are also canonically identified. Let $\cU = \cU_\xi$, for some $\xi \in A \ts B \ts U$.
  We prove that $\HH^1(X,\cU^*)=0$ if $\cU$ has no copy of $\cL$
  as direct summand, the other statement is analogous.
  We dualize \eqref{ulrich-extension}:
  \begin{equation}
    \label{dualE}
    0 \to B \ts \cL^* \to \cU^* \to A^* \ts \cO_X(F) \to 0.    
  \end{equation}
  Cup product gives a map:
  \[
  \HH^0(X,A^* \ts \cO_X(F))  \ts \Ext^1_X(B \ts \cL,A \ts \cO_X(-F))
  \to  \HH^1(X,B \ts \cL^*).
  \]
  Restricting to $\langle \xi \rangle \subset \Ext^1_X(B \ts \cL,A
  \ts \cO_X(-F))$ we get a map $\hat \xi : A^* \ts U \to B$ which is the boundary map
  associated with global
  sections of \eqref{dualE}, and we have:
  \[
  A^* \ts U \xr{\hat \xi} B   \to \HH^1(X,\cU^*) \to 0.
  \]
  
  Now if $\HH^1(X,\cU^*) \ne 0$ then
  the map $\hat \xi$ is not surjective, say it factors through
  $\notB^* \ts U \xr{\hat {\xi'}} B'$ with $B = B' \oplus
  B''$ and $B'' \ne 0$. Then $\xi$ can be written as $\xi' \in A \ts
  B' \ts U$, extended by zero to $A \ts B \ts U$. It follows that
  $B'' \ts \cL$ is a direct summand of $\cU^*$, which
  contradicts the assumption. The proof of the second statement
  follows the same pattern.
\end{proof}

\begin{lem} \label{noext}
Let $\cE$ be an ACM bundle on $X$, $P$
and $Q$ as in Lemma \ref{lemmetto}, and $\cU$ be an Ulrich bundle. If $\cU$ does not contain $\cL$ as a direct
summand, then $\Ext^1_X(\cU,P)=0$. If $\cU$ does not contain
$\cO_X(-F)$ as a direct summand, then $\Ext^1_X(Q,\cU)=0$.
\end{lem}

\begin{proof}
  Again we prove only the first statement, the second one being
  similar. So assume $\cU$ does not have $\cL$ as a direct summand.
  We first show:
  \begin{equation}
    \label{firstvanishing}
  \Ext^2_X(\cU,\cH^{-1} \Theta \Theta^*\cE)=0.    
  \end{equation}

  To check this, using \eqref{dominatess} it suffices to show $\Ext^2_X(\cU,\cL)=0$.
  In turn, writing $\cU$ in the form of Proposition  \ref{ulrichprop},
  we are reduced to check $\Ext^2_X(\cL,\cL)=0$ and
  $\Ext^2_X(\cO_X(-F),\cL)=0$, which are obvious. So \eqref{firstvanishing} is proved.

  We now look at the sheaf $I$ of \S \ref{ab}. In view of \eqref{firstvanishing}, in order to prove our statement
  it is enough to apply $\Hom_X(\cU,-)$ to \eqref{Ps} and verify
  $\Ext^1_X(\cU,I)=0$. To see this vanishing, we apply $\Hom_X(\cU,-)$
  to the exact sequence \eqref{I} defining $I$. Then we have to prove:
  \[
  \Ext^1_X(\cU,\cO_X)=\Ext^2_X(\cU,\cO_X(-F))=0.    
  \]
  The first vanishing is given by Lemma \ref{not containing}. The
  second follows again immediately by Proposition  \ref{ulrichprop} by
  applying  $\Hom_X(-,\cO_X(-F))$.
\end{proof}

The lemma on rigid extensions that we will use is the following
standard fact. We provide a proof for the reader's convenience.

\begin{lem} \label{varipezzi}
Let $c,d \ge 0$ be integers, $C$ and $D$ be vector spaces of
dimension $c$ and $d$. Consider two sheaves $\cA$ and $\cB$ with
 $\ext^1_X(\cB,\cA)=1$ and corresponding extension sheaf $\cGG$.
If $\cE$ fitting into
\[
0 \to C \ts \cA  \to \cE \to D \ts \cB \to 0,
\]
 is defined by $\eta \in \Ext^1_X(D \ts \cB,C \ts \cA)
\simeq C \ts D^*$ corresponding to a map $D \to C$ of rank $r$,
then:
\[
\cE \simeq \cGG^r \oplus \cA^{\max(c-r,0)} \oplus \cB^{\max(d-r,0)}.
\]
\end{lem}

\begin{proof}
  Apply $\Hom_X(\cB,-)$ to the exact sequence defining $\cE$ and
  restrict to $D \ts \langle \id_\cB \rangle$. We obtain this way a
  map $D \to C \ts \Ext^1_X(\cB,\cA)$ which is identified with the linear
  map $D \to C$ associated with $\eta$ after choosing a generator of $\Ext^1_X(\cB,\cA)$.
  In a suitable basis, this map is just a diagonal matrix and 
  each non-zero entry of this matrix corresponds to an instance of
  the generator $\Ext^1_X(\cB,\cA)$, hence provides
  a copy of $\cGG$. Completing to a basis of $D$ and $C$ gives
  the remaining copies of $\cB$ and $\cA$.
\end{proof}

\subsubsection{The reduced monad}

Let again $\cE$ be ACM on $X$. We write
its Ulrich part $\cU$ as:
\[
\cU = \cU_0 \oplus \cL^{b_0 }\oplus \cO_X(-F)^{a_0 },
\]
where $\cU_0$ contains no copy of $\cL$ or $\cO_X(-F)$ as direct factor.
We apply Lemma \ref{noext} to the summands of $\cU$ and use the display of the monad
\eqref{diagram-monad-2}, whereby getting subbundles $\cK_0\subset \cK$ and
$\cC_0\subset \cC$ with
decompositions:
\begin{align*}
\cK & = \cK_0 \oplus \cU_0 \oplus \cO_X(-F)^{a_0},  \\
\cC & = \cC_0 \oplus \cU_0 \oplus \cL^{b_0}.
\end{align*}
Here, the bundles $\cK_0$ and $\cC_0$ fit into:
\begin{align}
   & \label{K0} 0 \to P \to \cK_0 \to \cL^{b_0} \to 0, \\
   & \label{C0} 0 \to \cO_X(-F)^{a_0} \to \cC_0 \to Q \to 0.
\end{align}
In turn, again by Lemma \ref{noext}, we obtain a splitting:
\[
\cE = \cU_0 \oplus \cE_0.
\]
where $\cE_0$ is an extension of $\cK_0$ and $Q$, or equivalently of  $P$ and $\cC_0$.

Therefore, if $\cE$ is indecomposable ACM and not Ulrich, then $\cU_0=0$ and $a=a_0$,
$b=b_0$. In this case $\cU \simeq \cO_X(-F)^a \oplus \cL^b$, and we
speak of the {\it reduced monad}.

\begin{lem} \label{abc}
  Let $\cE$ be an indecomposable non-Ulrich ACM bundle on $X$. 
  \begin{enumerate}[i)]
  \item \label{a} For any twist of $\cE$ one has either $a=0$ and $Q=0$, or $b=0$ and $P=0$. 
  \item \label{b} If $\cE$ is initialized, the first case of the two
    cases above occurs.
  \item \label{c} If $\cE$ is initialized with $b_{0,0}=0$ then $\cE$ is $\cV$
    or $\cO_X(iF)$ with $0 \le i \le 3$.
  \end{enumerate}
\end{lem}

\begin{proof}
  We prove \eqref{a}. We have just seen that the Ulrich part of
  $\cE$ is $\cO_X(-F)^a \oplus \cL^b$. Let us show:
  \[
  \Ext^1_X(Q,\cK_0)=0.
  \]
  We know by Lemma \ref{s'annullas} that $\Ext^1_X(Q,P)=0$, so by
  \eqref{K0} it
  suffices to check $\Ext^1_X(Q,\cL)=0$. But this follows from Lemma
  \ref{noext} since $\cL$ is Ulrich.

  Looking at the reduced monad, we have proved $\cE \simeq \cK_0
  \oplus \cC_0$. By indecomposability of $\cE$, one of these two bundles must be zero, which
  is precisely what we need for \eqref{a}.

  Part \eqref{b} follows, since $\HH^0(X,\cC_0)=\HH^0(X,Q)=0$ by
  \eqref{J}, \eqref{Qs} and \eqref{C0}. 

  For \eqref{c},
  in view of \eqref{dominatess} and \eqref{Ps}, $b_{0,0}=0$ implies $P
  \simeq I \simeq \bigoplus_{i \ge 0}\cO_X(iF)^{p_i}$, where the
  second isomorphism is 
   \eqref{Isplits}. We would then have an exact sequence:
   \[
   0 \to \bigoplus_{i \ge 0}\cO_X(iF)^{p_i} \to \cE \to \cL^b \to 0,
   \]

  Note that any $i \ge 4$ is actually forbidden. Indeed,
  $\HH^1(X,\cO_X(-2H+iF)) \ne 0$ for $i \ge 4$. In turn this would easily imply
  $\HH^1(X,\cE(-2H)) \ne 0$, which is absurd for $\cE$ is
  ACM. Moreover $\ext_X^1(\cL,\cO_X(iF))$ is zero for $i \le 3$ and
  $1$ for $i=3$.
  This, together with Lemma \ref{varipezzi}, implies \eqref{c}.
\end{proof}

We need also another vanishing, this time for the bundle $\cV$ of
Example \ref{V-W}.

\begin{lem} \label{nienteV}
  The bundle $\cV$ satisfies $\Ext^1_X(\cL,\cV(H-F))=0$.
\end{lem}

\begin{proof}
Applying $\Hom_X(\cL(F-H),-)$ to the sequence of Example \ref{V-W}
defining $\cV$ we see that the space under consideration is the
cokernel of the linear map:
\[
\rho : \HH^0(X,\cO_X(H-F)) \to \HH^1(X,\cL^*(H-F))
\]
coming from cup product with the generator $\delta$ of $\HH^1(X,\cL^*)$.

We have to prove that this map is surjective. To see it, consider its
effect on a non-zero element corresponding to a curve $C \subset X$ of
class $H-F$ corresponding to a global section $s$ of $\cO_X(H-F)$. The curve $C$ gives:
\[
0 \to \cL^* \to \cL^*(H-F) \to \cL^*(H-F)|_C \to 0,
\]
and taking cohomology  we get a map $\sigma : \HH^1(X,\cL^*)
\to \HH^1(X,\cL^*(H-F))$, which sends $\delta$ to the image of $s$ via
$\rho$. Also $\HH^1(C,\cL^*(H-F)|_C)=0$ since $C$ is rational and $C \cdot \cL^*(H-F) > 0$.

But working over the exceptional curve $\Delta$, i.e. on the generator of
$\HH^0(X,\cL^*)$, we easily see that
$\hh^1(X,\cL^*(H-F))=1$. Therefore $\sigma$ is surjective (actually an
isomorphism) so $\delta$ has non-zero
image in $\HH^1(X,\cL^*(H-F))$. This says that $\rho$ is surjective.
\end{proof}

\subsubsection{The second reduction of the monad}

In the hypothesis of Lemma \ref{abc}, part \eqref{b}, we proved that $\cE \simeq
\cK_0$ where $\cK_0$ fits into \eqref{K0} and in turn $P$ fits into:
\[
0 \to \cL^{b_{0,0}} \to I \to P \to 0.
\]
Recall also the form \eqref{Isplits} of $I$. In the next lemma we carry out
the final step in the study of $P$.

\begin{lem} \label{reduction}
  Let $\cE$ be an indecomposable initialized ACM non-Ulrich bundle of rank $r>1$
  on $X$ with $b_{0,0} \ne 0$. Then $P$ is a direct sum of copies of
  $\cV(H-F)$, $\cO_X(H-F)$ and $\cO_X$.
\end{lem}

\begin{proof}
  We know by Lemma \ref{abc} that $a=0$ and $Q=0$. Define $\cE'=P(-H)$.
  Observe that, since $\cE$ is ACM and $\HH^0(X,\cU)=0$,
  from the monad we
  get $\HH^1(X,\cE'(tH)) = 0$ for $t\le 0$. 
  This vanishing allows to apply the construction of
  \S \ref{ab} to $\cE'$. We can then
  define the corresponding integers $a'_{i,j}$, $a'$, $b'_{i,j}$ and
  $b'$, together with the associated sheaves $Q'$ and $J'$.
  It is now easy to compute:
  \[
  a'_{0,0}= a'_{0,1} = a'_{2,0} = a'_{2,1} = 0, \qquad a'_{1,1}=b_{0,0}.
  \]

  We rely on exact sequences analogous to \eqref{per Qs}, \eqref{J} and \eqref{Qs} to deduce that $\cE'$ fist into:
  \[
  0 \to \cO_X(-F)^{b_{0,0}} \to \cE' \to Q' \to 0,
  \]
  with $Q'$ appearing as kernel:
  \[
  0 \to Q' \to \cL(-F)^{b'_{2,1}} \to \cL^{b'_{2,0}} \to 0.
  \]
  This implies that there are integers $q_j$ such that $Q'$ takes the form:
  \[
  Q' \simeq \bigoplus_{j \ge 1} \cL(-jF)^{q_j}.
  \]

  Now $\ext^1_X(\cL(-jF),\cO_X(-F))$ equals $0$
  for $j \ge 2$, and $1$ for $j = 1$. So by Lemma
  \ref{varipezzi} $\cE'$ is a direct sum of copies of $\cO_X(-F)$, 
  $\cV(-F)$ and $\cL(-jF)$ for $j \ge 1$. Actually
  $\HH^1(X,\cL(-H-jF))\ne 0$ for $j\ge 5$ implies that only $j \le 4$
  can occur. Note that $\cO_X(-F)$ or
  $\cV(-F)$ must occur because $a'_{1,1} \ne 0$.

  We now add $H$ and use indecomposability of $\cE$, after
  observing that $\Ext^1_X(\cL,\cL(H-jF))=0$ for $j \le 2$.
  Because $\cE$ is not a line bundle, this implies that only $j=3,4$
  can occur in the decomposition of $P$. But $\cL(H-4F) \simeq
  \cO_X(-F)$ does not appear because Lemma \ref{abc} says that $a=0$.
  Summing up, besides $\cV(H-F)$ and $\cO_X(H-F)$, only $j=3$ is allowed, giving $\cO_X$.
\end{proof}

\subsubsection{Proof of Theorem \ref{S(1,3)-is-tame}}

After the second reduction we are in position to prove that $S(1,3)$
is of tame CM type. So let $\cE$ be an indecomposable ACM bundle on
$X=S(1,3)$. In case $\cE$ is Ulrich, Proposition \ref{prop-ulrich}
shows assertions \eqref{ocio} and \eqref{occhio2}, so we only have to
take case of the non-Ulrich case.

We can  assume that $\cE$ is initialized of rank $r>1$, so that Lemma
\ref{reduction} applies. Then $\cE$ fits into:
\[
0 \to P \to \cE \to \cL^b \to 0,
\]
where $P$ is a direct sum of copies of $\cV(H-F)$, $\cO_X(H-F)$ and $\cO_X$.

By Lemma \ref{nienteV}, since $\cE$ is indecomposable we get that
either $\cE \simeq \cV(H-F)$, or $\cV(H-F)$ does not occur in $P$. We study this
second case again by means of a monad. Indeed, factoring out all
copies of $\cO_X(H-F)$ from $P$ we write $\cE$ as an extension by
$\cL^b$ of a quotient bundle $Q''$ of $\cE$ which we 
express as extension of copies of $\cO_X$ and $\cL$. This quotient bundle is thus a direct 
sum of copies of $\cO_X$, $\cL$ and $\cV$ by Lemma \ref{varipezzi}.

But again $\Ext^1_X(\cO_X,\cO_X(H-F))=0$ so in fact $\cO_X$ does not
occur as direct summand of $Q''$. Moreover, since $\cV$ has rank $2$ with $\wedge^2 \cV \simeq
\cL$ we have $\cV \ts \cL^* \simeq \cV$, so from Lemma \ref{nienteV}
we deduce $\Ext^1_X(\cV,\cO_X(H-F))=0$. Hence either $\cE$ is $\cV$, or
$\cV$ does not appear either as direct summand of $Q''$.

We have proved that $\cE$ is then an extension of copies of $\cO_X(H-F)$ and $\cL$,
which is hence necessarily isomorphic to $\cW$, again by Lemma \ref{varipezzi}.
The proof of Theorem \ref{S(1,3)-is-tame} is now complete.

\begin{rmk} 
  Theorem \ref{quartic-is-tame} (and hence Theorem \ref{main-tame}) hold, except for part
  \eqref{occhio}, over an arbitrary field $\bk$. In turn, 
  if $\bk$ is not
  algebraically closed, part  \eqref{occhio} holds for geometrically
  indecomposable $\cE$ (i.e., $\cE$ is indecomposable over the
  algebraic closure of $\bk$).
  Otherwise, indecomposable ACM bundles over $\bk$ are obtained from
  companion block matrices associated with irreducible polynomials in
  one variable over $\bk$, cf. again \cite{burgisser-clausen-shokrollahi}.
\end{rmk}

\section{Rigid bundles on scrolls of higher degree}

\label{rigid section}

In this section we set $X=\FF_\epsilon$ with $\epsilon = 0,1$, so
 $X$ is a del Pezzo surface, and we assume $d_X \ge 5$, hence
 $X$ is of wild representation type.
For the rest of the paper, we assume $\mathrm{char}(\bk)=0$.

\subsection{Construction of rigid ACM bundles} \label{rigid construction}

Let $s \ge 0$ be an integer, and consider a vector of $s$ 
 integers
 $\vk=(k_1,\ldots,k_s)$.
With $\vk$ we associate the word $\sigma^\vk$ of the braid group $B_4$ by:
\[
\sigma^\emptyset = 1, \quad \sigma^{\vk} = \sigma_1^{k_1} \sigma_2^{k_2} \sigma_1^{k_3}
\sigma_2^{k_4} \cdots \in B_3 = \langle \sigma_1,\sigma_2 \mid
\sigma_1\sigma_2\sigma_1 = \sigma_2\sigma_1\sigma_2\rangle.
\]
This means that $\sigma^{\vk}$ belongs to the copy of $B_3$ in $B_4$
consisting of braids not involving the first strand.
Clearly, up to adding $k_{i-1}$ and $k_{i+1}$, we may assume $k_i \ne 0$, for $i>1$. 

This subgroup operates on $3$-terms exceptional collections over $X$, in view of the
$B_4$-action on full exceptional collections described in
\S \ref{braid-definition}. Set $\bB_\emptyset$ for the subcollection
$(\cL[-1],\cO_X(-F),\cO_X)$ of $\bC_\emptyset$, cf. \eqref{C1}.
Put $\bB_\vk=\sigma^\vk \bB_\emptyset$ for the subcollection
of $\bC_\vk=\sigma^\vk\bC_\emptyset$ obtained by mutation via $\sigma^{\vk}$.
Given a $3$-term exceptional collection $\bB=(\cS_1,\cS_2,\cS_3)$ we
consider:
\[
v =  (\chi(\cS_2,\cS_3),\chi(\cS_1,\cS_3),\chi(\cS_1,\cS_2)) \in \Z^3.
\]
The group $B_3$ thus operate on $\Z^3$, by sending a vector
$v=(v_1,v_2,v_3)$ to:
\begin{align*}
  \sigma_1 : (v_1,v_2,v_3) \mapsto (v_1v_3-v_2,v_1,v_3), &&
  \sigma_2 : (v_1,v_2,v_3) \mapsto (v_1,v_1v_2-v_3,v_2).
\end{align*}
This action factors through the center of $B_3$, so that it
actually defines an operation of the modular group
$\mathrm{PSL}(2,\Z)$ on $\Z^3$.

The inverse of $\sigma_1$ and $\sigma_2$ operate by similar formulas.
For the basic collection
$\bB_\emptyset$ we have:
\[
v=v^\emptyset = (2,d_X-4,d_X-2).
\]
We set $v^\vk \in \Z^3$ for the vector corresponding to $\bB_\vk$, 
i.e. $v^\vk=\sigma^\vk.v^\emptyset$.

Next, we set $\bar t \in \{0,1\}$ for the remainder of the
division of an integer $t$ by $2$.
Given $\vk=(k_1,\ldots,k_s)$ and $t\le s$, we write the truncation
$\vk(t)=(k_1,\ldots,k_{t-1})$.
We define:
\begin{equation}
  \label{positive}
  \fbK = \{\vk = (k_1,\ldots,k_s) \mid (-1)^{t} k_{t-1}
  (v^{\vk(t)})_{2 \bar t +1} \le 0, \forall t \le s\}.
\end{equation}
Note that, if $\vk \in \fbK$, then $\vk(t)\in \fbK$,  $\forall t\le s$.
Observe that belonging to $\fbK$ imposes no restriction on the last
coordinate $k_s$ of $\vk$.

Next we provide an existence result for rigid ACM
bundles indexed by $\vk \in \fbK$. 
There might be $\vk \ne \vk'$ such that $\sigma_\vk=\sigma_{\vk'}$ in
view of the Braid group relation $\sigma_1\sigma_2\sigma_1 = \sigma_2\sigma_1\sigma_2$.
The correspondence described by the next theorem has to be understood up to this
ambiguity. The next result proves the existence part of Theorem
\ref{alla fine je l'abbiamo fatta}.

\begin{thm} \label{Fk}
  To any $\vk$ in $\fbK$ there corresponds 
  an exceptional
  ACM bundle, denoted by $\fF_{\vk}$, sitting as middle term of the exceptional
  collection $\bB_\vk$ obtained  from
  $\bB_\emptyset$ by mutation via
  $\sigma^{\vk}$. 
\end{thm}

\begin{proof}
We first give a step-by-step algorithm to construct $\fF_{\vk}$.
\begin{step} \label{step1}
  Take $s=1$ and start with $\bB=\bB_\emptyset$, cf. \eqref{C1}. 
  Observe  that the
  exceptional pair $(\cS_1,\cS_2)$ is irregular, with
  $\cS_1=|\cS_1|[-1]$, $\cS_2=|\cS_2|$,
  $(\sigma^1.v)_3=\chi(\cS_1,\cS_2) = d_X-2 > 0$. 
  Consider then the objects $\cG_{k}$ in the notation of \S
  \ref{braid-definition}, and set
  $\fF_{\vk} = |\cG_{k_1+1}|$.
  In the notation of \S \ref{ulrich-notation}, we have
  $\fF_{k_1}\simeq \fU_{k_1+1}$, i.e. the case $s=1$ corresponds to Ulrich bundles.

  To prepare the next step, modify $\bB$ by keeping 
  $\cS_3$ unchanged, but replacing $\cS_1$ with $\cG_{k_1}$ and $\cS_2$ with $\cG_{k_1+1}$.
  We observed that \eqref{positive} imposes no condition for
  $s=1$ so $k_1$ is arbitrary in $\Z$ if $s=1$.  
  On the other hand, for $s\ge 2$, taking $t=2$
  leads to assume $k_1(\sigma^{\vk(2)}.v)_1 \le 0$. One can check right away
  that this simply
  means $k_1 \le 0$.
\end{step}  

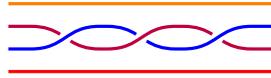
\vspace{-0.5cm}
\begin{center}
  \begin{figure}[h!]
      \begin{tikzpicture}
        \braid[very thick, rotate=90, width = 0.3cm, height = 1.0cm, number of strands = 4, style strands={1}{red}, style
        strands={2}{blue},style strands={3}{purple}, style strands={4}{orange}]
        s_2 s_2 s_2;
      \end{tikzpicture}
  \end{figure}
\end{center}

\vspace{-0.5cm}
For instance, take $\vk = (-3)$, i.e. $\sigma=\sigma_1^{-3}$, so $\fF_{-3}=\fU_{-2}$. We draw
the corresponding braid of $B_4$, and the associated mutations,
starting from the basic exceptional collection $\bB_\emptyset$:
\begin{align*}
   \sigma_1^{-1} \bB_\emptyset = \bB_{-1} & =(\fU_{-1}[-1],\fU_{0}[-1],\cO_X), && \mbox{for $k_1=-1$}; \\
   \bB_{-2}& =(\fU_{-2}[-1],\fU_{-1}[-1],\cO_X), && \mbox{for $k_1=-2$}; \\
   \bB_{-3}& =(\fU_{-3}[-1],\fU_{-2}[-1],\cO_X), && \mbox{for $k_1=-3$}. 
\end{align*}

\begin{step} 
  Take $s\ge 2$ even. 
  It turns out (cf. Lemma \ref{it's
    irregular}) that the condition $\vk \in \fbK$ implies that, in the collection $\bB$ defined inductively by
  the truncation $\sigma^{\vk(s)}$,
  the pair $(\cS_2,\cS_3)$ is irregular. Then, it can be used to construct the objects $\cG_{k}$ as in \S\ref{braid-definition}.
  Next, we define $\fF_{\vk} = |\cG_{k_s}|$ and modify $\bB$ by keeping $\cS_1$ unchanged
   and replacing the pair $(\cS_2,\cS_3)$ with
   $(\cG_{k_s},\cG_{k_s+1})$.
\end{step}

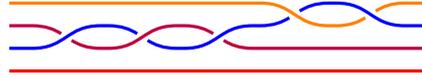
\vspace{-0.2cm}
\begin{center}
  \begin{figure}[h!]
      \begin{tikzpicture}
        \braid[very thick, rotate=90, width = 0.3cm, height = 1cm, number of strands = 4, style strands={1}{red}, style
        strands={2}{blue},style strands={3}{purple}, style strands={4}{orange}]
        s_2 s_2 s_2 s_3^{-1} s_3^{-1};
      \end{tikzpicture}
  \end{figure}
\end{center}

\vspace{-0.5cm}
Take e.g. $\sigma=\sigma_1^{-3}\sigma_2^2$, i.e. $\vk =
(-3,2)$. The associated mutations start from $\bB_{-3}$ and give:
\begin{align*}
  & \bB_{-3}=(\fF_{-4}[-1],\fF_{-3}[-1],\cO_X), && \mbox{for
    $\vk=(-3)$}, && \cO_X=\fF_{-3,1};  \\
  & \bB_{-3,1}=(\fF_{-4}[-1],\fF_{-3,1},\fF_{-3,2}), && \mbox{for    $\vk=(-3,1)$}; \\
  & \bB_{-3,2}=(\fF_{-4}[-1],\fF_{-3,2},\fF_{-3,3}), && \mbox{for  $\vk=(-3,2)$}.
\end{align*}

\begin{step} 
  Take $s \ge 3$ odd. Again apply Lemma \ref{it's
    irregular} to the collection $\bC_{\vk(s)}$.
  This time the irregular pair is $(\cS_1,\cS_2)$, and we use it 
  to construct the objects $\cG_{k}$, cf. \S\ref{braid-definition}.
  Put $\fF_{\vk} = |\cG_{k_s+1}|$.
  As in Step \ref{step1}, modify the collection $\bB$ by keeping 
  $\cS_3$ and replacing $\cS_1$ with $\cG_{k_s}$ and $\cS_2$ with $\cG_{k_s+1}$.
\end{step}

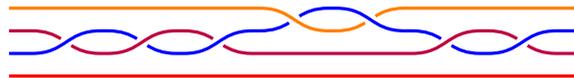
\vspace{-0.4cm}
\begin{center}
  \begin{figure}[h!]
      \begin{tikzpicture}
        \braid[very thick, rotate=90, width = 0.3cm, height = 1cm, number of strands = 4, style strands={1}{red}, style
        strands={2}{blue},style strands={3}{purple}, style strands={4}{orange}]
        s_2 s_2 s_2 s_3^{-1} s_3^{-1} s_2 s_2;
      \end{tikzpicture}
  \end{figure}
\end{center}

\vspace{-0.5cm}
Take for example $\vk=(-3,2,-2)$. This lies in $\fbK$ for $X=S(2,3)$. 
The bundle $\fF=\fF_{-3,2,-1}$ has canonical slope
$c_1(\fF) \cdot c_1(\omega_X)/\rk(\fF)=37/216$. The associated mutations give:
\begin{align*}
  & \bB_{-3,2}=(\fF_{-4}[-1],\fF_{-3,2},\fF_{-3,3}), && \mbox{for $\vk=(-3,2)$}; \\
  & \bB_{-3,2,-1}=(\fF_{-3,2,-1}[-1],\fF_{-4}[-1],\fF_{-3,3}), && \mbox{for $\vk=(-3,2,-1)$}; \\
  & \bB_{-3,2,-2}=(\fF_{-3,2,-2}[-1],\fF_{-3,2,-1}[-1],\fF_{-3,3}), && \mbox{for $\vk=(-3,2,-2)$}.
\end{align*}

Having this in mind, it is clear that $\fF_{\vk}$ is exceptional,
by the basic properties of mutations. Also, by induction on $s$ we see
that $\fF_\vk$ is an ACM bundle, as it is an extension of ACM bundles, in view of the Fibonacci sequences \eqref{NP>} and \eqref{NP<}.
To conclude we need the following lemma.
\end{proof} 

\begin{lem}  \label{it's irregular}
Let $\vk \in \fbK$ and let $k_{s+1} \in \Z \setminus \{0\}$.
Set $\bB=\bB_{\vk}=(\cS_1,\cS_2,\cS_3)$.
For $s$ even, $(\cS_{1},\cS_{2})$ is an irregular
pair if and only if:
\begin{enumerate}
\item \label{tutto<} $k_{s}<0$, $(\sigma^{\vk}.v)_3<0$, in
  which case $\cS_i=|\cS_i|[-1]$ for all $i$, or:
\item $k_{s}>0$, $(\sigma^{\vk}.v)_3>0$, so
   $\cS_1=|\cS_1|[-1]$, $\cS_i=|\cS_i|$ for $i=2,3$.
\suspend{enumerate}
For $s$ odd, $(\cS_{2},\cS_{3})$ is an irregular
pair if and only if:
\resume{enumerate}
\item $k_{s}<0$, $(\sigma^{\vk}.v)_1>0$, 
  so $\cS_3=|\cS_3|$,  $\cS_i=|\cS_i|[-1]$ for $i=1,2$;
\item \label{tutto>} $k_{s}>0$ and $(\sigma^{\vk}.v)_1<0$, in
  which case $\cS_i=|\cS_i| $ for all $i$.
\end{enumerate}
Finally, $(k_1,\ldots,k_{s+1})$ lies in $\fbK$ if and only if one the
cases \eqref{tutto<} to \eqref{tutto>} occurs.
\end{lem}

  The inequalities of the cases from \eqref{tutto<} to
  \eqref{tutto>} simply correspond to the conditions on $\vk$ given by
  \eqref{positive}, so of course one (and only one) of them is
  verified if and only if $(k_1,\ldots,k_{s+1})$ lies in $\fbK$.

\begin{proof}[Proof of Lemma \ref{it's irregular}]
  Let $s$ be even. Note that, as soon as the objects 
  $\cS_i$ are concentrated in the  
  degree prescribed by the sign of $k_s$, we see that the sign of
  $(\sigma^{\vk(s)}.v)_3$ corresponds precisely to 
  $(\cS_1,\cS_2)$ being a regular or irregular exceptional pair.
  The same happens for odd $s$. Therefore it suffices to check that
  $\cS_1,\cS_2,\cS_3$ are concentrated in the correct degrees, as listed in 
  \eqref{tutto<} to \eqref{tutto>}.

  In turn, this last fact is easy to see by induction on $s$. Indeed,
  each time we pass from $t \le s$ to $t+1$, we operate 
  mutations based on an irregular exceptional pair of bundles
  $(\cP,\cN)=(\cS_{\bar t+1},\cS_{\bar t+2})$.
  So, in the notation of \S \ref{braid-definition}, we set
   $w = \ext^1_X(|\cP|,|\cN|)$, and we see that the bundles
  $\cG_{k_{t+1}}$ and $\cG_{k_{t+1}+1}$ fit as middle term of an
  extension of the form \eqref{NP>} (for 
  $k_{t+1}>0$) or \eqref{NP<} (for  $k_{t+1}<0$). This implies that
  $\cS_1,\cS_2,\cS_3$ are concentrated in the desired degrees.
\end{proof}

\begin{rmk}
  One can define similarly an action of $B_3=\langle \sigma_0,\sigma_1
  \mid \sigma_0\sigma_1\sigma_0 = \sigma_0 \sigma_1\sigma_0\rangle
  \subset B_4$ on the set of three-terms exceptional collections leaving $\cO_X$ fixed. 
  Write $\fH_\vk$ for the bundle associated  in this sense with the word
  $\sigma_0^{k_1} \sigma_1^{k_2}\sigma_0^{k_3} \cdots$. Then we have
  $\fH_{\vk}(-H) \simeq \fF_{-\vk}^* \ts \omega_X$.
  Here is the braid for $\fH_{3,-2,2}$, which has canonical
  slope $253/216$:

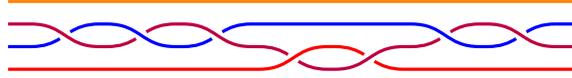
\vspace{-0.2cm}

\begin{center}
  \begin{figure}[h!]
      \begin{tikzpicture}
        \braid[very thick, rotate=90, width = 0.3cm, height = 1cm, number of strands = 4, style strands={1}{red}, style
        strands={2}{blue},style strands={3}{purple}, style strands={4}{orange}]
        s_2^{-1} s_2^{-1} s_2^{-1} s_1 s_1 s_2^{-1} s_2^{-1};
      \end{tikzpicture}
  \end{figure}
\end{center}

Also, if $d_X \ge 6$ and $k_1 < 0$, or if $k_1 < -
2$,  it is possible to prove, by induction on $s$, that any $\vk = (k_1,\ldots,k_s)$ with $k_{t-1}k_{t}<0$ for all $t<s$
lies in $\fbK$.  
\end{rmk}

\subsection{Classification of rigid ACM bundles on two rational normal
  scrolls}

Here we assume that $X$ is $S(2,3)$ or $S(3,3)$. So $X$ is a del Pezzo
surface embedded as a CM-wild rational normal scroll of degree $5$ or
$6$.
The next results classifies indecomposable rigid ACM bundles
on $X$, whereby concluding the proof of Theorem \ref{alla fine je
  l'abbiamo fatta}.

\begin{thm} \label{we classify}
  Let $X=S(2,3)$ or $X=S(3,3)$, and let $\cE$ be an indecomposable rigid ACM
  bundle on $X$. Then  there exists
   $\vk = (k_1,\ldots,k_s)$ in $\fbK$ such that, up to
  a twist, $\cE \simeq
  \fF_{\vk}$ or $\cE ^* \ts \omega_X \simeq \fF_{\vk}$.
\end{thm}

\begin{lem} \label{all rigid}
  Let $(\cP,\cN)$ be an exceptional pair of bundles on $X$, and let $p$
  and $q$ be integers. Let $\cE$ be a bundle on $X$ equipped with two maps
  $f,g$ as in the diagram: 
  \begin{equation}
    \label{monadfg}
      |\cN|^p \xr{f}   \cE \xr{g}   |\cP|^q, \qquad \mbox{with $g \circ
    f = 0$, $f$ injective, $g$ surjective}.
  \end{equation}
 Write $\cK = \ker(g)$,
  $\cC = \coker(f)$, and $\cF=\cK/|\cN|^p$.
  Finally, assume:
  \begin{align}
    \label{miparecivoglia} & \RHom_X(\cN,\cF) = 0, &&    \RHom_X(\cF,\cP)=0, \\
    \label{forseanchequesto} & \Ext^2_X(|\cP|,\cF)=0, && \Ext^2_X(\cF,|\cN|)=0.
  \end{align}

  Then, whenever $\cE$ is rigid, also $\cK$, $\cC$ and $\cF$ are rigid.
\end{lem}

\begin{proof}
Write the following exact commutative diagram as display of \eqref{monadfg}.
\begin{equation}\label{monad2}
   \xymatrix@-3ex{& 0 \ar[d] & 0 \ar[d]\\
    & |\cN|^p \ar@{=}[r] \ar[d] & |\cN|^p \ar[d]  \\
    0 \ar[r] & \cK \ar[r]  \ar[d]  & \cE \ar[r] \ar[d] & |\cP|^q \ar@{=}[d]\ar[r] & 0\\
    0 \ar[r] & \cF \ar[r] \ar[d] & \cC \ar[r] \ar[d] &  |\cP|^q \ar[r] & 0\\
    & 0 & 0 }
\end{equation}
  Applying $\Hom_X(-,\cF)$ to the leftmost column of \eqref{monad2} and
  using \eqref{miparecivoglia} we get:
  \[
  \Ext^i_X(\cF,\cF) \simeq \Ext^i_X(\cK,\cF),  \qquad \mbox {for all $i$}.
  \]
  By the same sequence, using $\Hom_X(-,|\cN|)$ and 
  \eqref{forseanchequesto}), we get  $\Ext^2_X(\cK,|\cN|)=0$ since
  $\cN$ is exceptional. Then, applying
  $\Hom_X(\cK,-)$ to the same sequence we get: 
  \[
  \Ext^1_X(\cK,\cK) \epi   \Ext^1_X(\cF,\cF).
  \]
  Then, to prove that $\cF$ is rigid, it suffices to prove that $\cK$
  is. To do it, we apply $\Hom_X(-,\cP)$ to the left column of
  \eqref{monad2}. Since $(\cP,\cN)$ is an exceptional pair, in view of
  \eqref{miparecivoglia} we get:
  $
  \RHom_X(\cK,\cP)=0.
  $
  Therefore, applying $\Hom_X(\cK,-)$ to  the central row of
  \eqref{monad2} we obtain:
  \begin{equation}
    \label{uguali KE}
  \Ext^i_X(\cK,\cK) \simeq \Ext^i_X(\cK,\cE),  \qquad \mbox {for all $i$}.    
  \end{equation}

  Note that, since $\Ext^2_X(|\cP|,|\cN|)=0$ (because $(\cP,\cN)$ is an
  exceptional pair, regular or not), from \eqref{forseanchequesto} we deduce, applying
  $\Hom_X(|\cP|,-)$ to \eqref{monad2}, the vanishing $\Ext^2_X(|\cP|,\cE)$.
  Therefore, applying $\Hom_X(-,\cE)$ to the central row of
  \eqref{monad2} we get a surjection:
  \[
  \Ext^1_X(\cE,\cE) \epi   \Ext^1_X(\cK,\cE).
  \]
  This, combined with \eqref{uguali KE}, says that $\cK$ (and
  therefore $\cF$) is rigid if
  $\cE$ is. A similar argument shows that the same happens to $\cC$.
\end{proof}

\begin{proof}[Proof of Theorem \ref{we classify}]
  Recall Proposition \ref{prop-abcd}, the notation $a,b,c,d$ and the construction of
  the monad associated with $\cE$, cf. \S \ref{monadology}. We have thus
  a kernel bundle $\cK$, a cokernel bundle $\cC$ and an
  Ulrich bundle $\cF$ associated with $\cE$. 

  We use now Lemma \ref{all rigid} with $\cP=\cL(-F)[-1]$, $p=c$,
  $\cN=\cO_X$ and $q=d$.
  Since $\cF$ is an extension of $\cL$ and $\cO_X(-F)$ and 
  $\bC_\emptyset$ is an exceptional collection we have the vanishing \eqref{miparecivoglia}.
  Also, for $\bC_\emptyset$ we have $\Ext^2_X(|\cS_i|,|\cS_j|)=0$ so
  \eqref{forseanchequesto} also holds.
  Therefore by Lemma \ref{all rigid} we see that $\cK$, $\cC$ and
  $\cF$ are all rigid.
  By Lemma \ref{aspita}, part \eqref{somma}, there must be 
   $k \ne 0$ and $a_k$, $a_{k+1}$ such that:
  \[
  \cF \simeq \fU_k^{a_k} \oplus \fU_{k+1}^{a_{k+1}},
  \]
  where $\fU_k$ are the exceptional Ulrich bundles, and
  $\fF_k=\fU_{k+1}$. Now, in Lemma \ref{it's irregular} we computed:
  \begin{align*}
  &\chi(\cL(-F),\fU_{k}) \ge 0, && \mbox{iff $k \le 0$}, \\
  &\chi(\fU_{k},\cO_X) \ge 0, && \mbox{iff $k \ge 1$}.    
  \end{align*}
  Up to possibly a shift by $1$, the exceptional pair $(\fU_k,\fU_{k+1})$ is
  completed to the full exceptional collection   
  $\bC_{k}$ by adding $\cL(-F)[-1]$ to the left and to the right $\cO_X$.
  So, by \cite[Proposition 5.3.5]{gorodentsev-kuleshov}, the previous display implies
  the vanishing:
  \begin{align}
  \label{justproved} & \Ext^1_X(\cL(-F),\cF) = 0, && \mbox{if $k < 0$}, \\
  \nonumber & \Ext^1_X(\cF,\cO_X) = 0, && \mbox{if $k > 0$}.
  \end{align}

  Let us consider the case $k < 0$. Before going further, note that analyzing this
  case is enough, indeed if $k>0$, then
  replacing $\cE$ with $\cE^* \ts \omega_X$ we get another rigid
  ACM bundle, whose Ulrich part will be this time in the range $k <
  0$.

  Assuming thus $k < 0$, by \eqref{justproved}, looking back at \eqref{diagram-monad}, we see that $\cC
  \simeq \cF \oplus \cL(-F)^d$. Then we note that, over $S(2,3)$ and
  $S(3,3)$, we have $\Ext^1_X(\cL(-F),\cO_X)=0$, so that $\cL(-F)^d$
  is a direct summand of $\cE$. So, since $\cE$ is indecomposable,
  we get $d=0$, or $\cE \simeq \cL(-F)$.
  In the second case, $\cE^*\ts \omega_X(H) \simeq \cO_X =
  \fF_{-1,1}$. Hence we assume $d=0$, so
  $\cE$ lies in the subcategory generated by the subcollection
  $\bB_{k}$ of $\bC_{k}$.
  Our goal is to show that $\cE$ can be constructed by the steps of
  Theorem \ref{Fk}, when we set $k_1=k$. This is clear for $\cF=0$, in which
  case $\cE \simeq \cO_X$. Also, it is clear for $c=0$,
  as we may assume $a_{k_1}=0$ and
  $a_{k_1+1}=1$ by indecomposability of $\cE$, so $\cE \simeq \fU_{k_1+1}
  \simeq \fF_{k_1}$. 

  To study the case $c\ne 0$, $\cF \ne 0$,
  first note that again indecomposability of $\cE$ forces
  $(\cS_2,\cS_3)$ to be an irregular pair, which gives back $k_1 < 0$
  and $\vk=(k_1,k_2) \in \fbK, \forall k_2 \ne 0$.
  Then, we use again Lemma \ref{all rigid}, this time with
  $\cN=\cO_X$, $p=c$, $\cP=\fU_{k_1}[-1]$ and $q=a_{k_1}$, so that the
  new bundle $\cF$ is $\fU_{k_1+1}^{a_{k_1+1}}$.
  We get a new bundle $\cK$ fitting into:
  \[
  0 \to \cO_X^c \to \cK \to \fU_{k_1+1}^{a_{k_1+1}} \to 0.
  \]
  Since $\bB_{k_1}=(\fU_{k_1}[-1],\fU_{k_1+1}[-1],\cO_X)$ is an exceptional sequence with
  vanishing $\Ext^2$ groups, Lemma \ref{all rigid} applies and shows that $\cK$ is
  rigid. So, again by  By Lemma \ref{aspita}, part
  \eqref{somma}, there are
  integers $k_2$, $a_{k_1,k_2}$ and $a_{k_1,k_2+1}$ such that, in the notation of Theorem \ref{Fk}:
  \[
  \cK \simeq \fF_{k_1,k_2}^{a_{k_1,k_2}} \oplus \fF_{k_1,k_2+1}^{a_{k_1,k_2+1}},
  \]

  Next, note that the bundle $\cE$ belongs to the subcategory generated
  by the new exceptional sequence
  $\bB_{k_1,k_2}$. Up to shifts, $\bB_{k_1,k_2}$ takes the form $(\fF_{k_1-1},\fF_{k_1,k_2},\fF_{k_1,k_2+1})$. If $a_{k_1}=0$ then
  $\cE \simeq \cK$ so since $\cE$ is indecomposable we see that $\cE$
  the form $\fF_{\vk}$, with $s=2$ in the notation of Theorem
  \ref{Fk}. 
  Otherwise, the new exceptional pair $(\cS_1,\cS_2)$ in $\bB_{k_1,k_2}$ must be
  irregular, so Lemma \ref{it's irregular}
  ensures that $\vk=(k_1,k_2,k_3)$ lies in $\fbK$ for any $k_3 \ne 0$. 
  Further, Lemma \ref{all rigid} can be used again, this time with
  $\cN=\fF_{k_1,k_2}$, $p=a_{k_1,k_2+1}$ and $\cP=\fF_{k_1-1}$, $q=a_{k_1}$ to see that the new
  bundle $\cC$ is
  rigid and hence a direct sum
  $\fF_{k_1,k_2,k_3}^{a_{k_1,k_2,k_3}} \oplus
  \fF_{k_1,k_2,k_3+1}^{a_{k_1,k_2,k_3+1}}$ for some $k_3 \ne 0$ and integers $a_{k_1,k_2,k_3}$, $a_{k_1,k_2,k_3+1}$.
  This process may be iterated, and eventually must stop
  because $\cE$ has finite rank. We finally get that the bundle $\cE$
  is of the form $\fF_{\vk}$ for some vector $\vk$ lying in $\fbK$.
\end{proof}

\noindent {\bf Acknowledgements}. We would like to thank the referee
for useful remarks.

\bibliographystyle{alpha}
\bibliography{bibliography}

\vfill

\end{document}